\documentclass[11pt,reqno]{amsproc}

\title[]{Well-posedness for SQG sharp fronts with unbounded curvature}

\author{Francisco Gancedo}
\address{Departamento de An\'alisis Matem\'atico $\&$ IMUS, Universidad de Sevilla, Sevilla, Spain}
\email{fgancedo@us.es}

\author{Huy Q. Nguyen}
\address{Department of Mathematics, Brown University, Providence, RI 02912}
\email{hnguyen@math.brown.edu}

\author{Neel Patel}
\address{Department of Mathematics, University of Michigan, Ann Arbor, MI 48109}
\email{neeljp@umich.edu}

\usepackage[margin=1in]{geometry}
\usepackage{amsmath, amsthm, amssymb, mathrsfs, stmaryrd}
\usepackage{times}
\usepackage{color}
\usepackage{hyperref}

\newcommand{\bq}{\begin{equation}}
	\newcommand{\eq}{\end{equation}}
\newcommand{\bqa}{\begin{eqnarray*}}
	\newcommand{\eqa}{\end{eqnarray*}}

\theoremstyle{plain}
\newtheorem{theo}{Theorem}[section]
\newtheorem{prop}[theo]{Proposition}
\newtheorem{lemm}[theo]{Lemma}

\newtheorem{defi}[theo]{Definition}
\theoremstyle{definition}
\newtheorem{rema}[theo]{Remark}

\DeclareMathOperator{\supp}{supp}

\DeclareSymbolFont{pletters}{OT1}{cmr}{m}{sl}
\DeclareMathSymbol{s}{\mathalpha}{pletters}{`s}

\def\tt{\theta}
\def\eps{\varepsilon}
\def\na{\nabla}

\def\les{\lesssim}

\def\mez{\frac{1}{2}}
\def\tdm{\frac{3}{2}}

\def\Rr{\mathbb{R}}
\def\T{\mathbb{T}}
\def\Nn{\mathbb{N}}
\def\Zz{\mathbb{Z}}
\def\Cc{\mathbb{C}}

\def\cF{\mathcal{F}}

\def\cN{\mathcal{N}}
\def\cP{\mathcal{P}}
\def\cT{\mathcal{T}}

\def\ld{\lambda}
\def\Ld{\lambda}
\def\p{\partial}
\def\na{\nabla}

\def\wt{\widetilde}
\def\Ld{\Lambda}
\def\g{\gamma}
\def\dg{\partial_\gamma}
\def\wh{\widehat}
\def\wsc{\overset{\ast}{\rightharpoonup}}

\numberwithin{equation}{section}

\newcommand{\e}{\eta}
\newcommand{\de}{\partial_{\eta}}

\date{today}
\begin{document}
	\begin{abstract}
Patch solutions for the surface quasigeostrophic (SQG) equation model sharp temperature fronts in atmospheric and oceanic flows. We establish local well-posedness for SQG sharp fronts of low Sobolev regularity, $H^{2+s}$ for arbitrarily small $s>0$,  allowing for fronts with unbounded curvature.
	\end{abstract}
	
	
	\noindent\thanks{\em{ MSC Classification:  .}}

	\maketitle	
\section{Introduction}
	We study the surface quasigeostrophic (SQG) equation 
\begin{align}\label{transport}
	&\p_t \tt(x, t)+u(x, t)\cdot \na \tt(x, t)=0,\quad (x, t)\in \Rr^2\times \Rr_+,\\ \label{SQG}
	&u=\na^\perp \Ld^{-1}\tt,\quad \Ld=(-\Delta)^\mez.
\end{align}
The SQG equation of geophysical significance comes from the 3D quasigeostrophic (QG) equation in the half-space. The QG equation arises in the study of stratified flows in which the Coriolis force is balanced with the pressure and serves as a model in simulations of large-scale atmospheric and oceanic circulation \cite{P}. In the ideal case of zero bulk vorticity, the QG equation reduces to a self-consistent equation on the boundary--the SQG equation \eqref{transport}-\eqref{SQG}.  SQG  shares many analogies with the 3D Euler equations \cite{CMT} and has been studied extensively. In particular, vortex lines for the 3D Euler equations are analogous to  levels sets of the scalar $\theta$.  

SQG is an active scalar equation, i.e. the  transport velocity is given in terms of the transported scalar. A class of active scalar equations, called generalized SQG (gSQG), is given by \eqref{SQG} with the velocity 
\bqa\label{gBS}
u=\na^\perp \Ld^{-2+\alpha}\tt,\quad \alpha\in [0, 1].
\eqa
The case $\alpha=0$ corresponds to the vorticity formulation of the 2D Euler equations. The scalar-to-velocity relation becomes more singular as $\alpha$ increases. 

While the SQG equation is locally well-posed in  Sobolev spaces $H^s(\Rr^2)$, $s>2$,  the question of global regularity versus finite-time singularity remains open \cite{CMT}. This is also the case for gSQG for any $\alpha\in (0, 1)$ \cite{CCW}. Nevertheless, certain blow-up scenarios have been ruled out; eg. collision of level sets in the development of sharp fronts \cite{CF}, hyperbolic saddle breakdown \cite{C}.
There exist families of non-trivial global solutions \cite{CCG-S2} but some small initial data provide growth for all time if lack of finite time singularities is assumed \cite{HK}. 
 On the other hand, $L^p$ weak solutions are known to exist globally, but their uniqueness is unknown. More precisely, by appealing to the commutator structure of SQG's nonlinearity in the weak form, global weak solutions were obtained for $L^2$ \cite{Resnick} and lower $L^p$ integrability ($p>4/3$) \cite{Marchand}. Nonunique weak solutions with negative Sobolev regularity were constructed in \cite{BSV}.  $L^p$ weak solutions include the important class of {\it patch solutions} which are given by the characteristic function of an evolving domain $D(t)\subset \Rr^2$:
\bq\label{front}
\tt(x, t)=\begin{cases}
\tt^0\quad\text{if } x\in D(t),\\
0\quad\text{if } x\notin D(t),
\end{cases}
\eq
where $\tt^0$ is a nonzero constant. In other words, $D(t)$ is the level set $\{\tt(t)=\tt^0\}$. Patch solutions of the 2D Euler equations are also called vortex  patches \cite{bertozzi-Majda}. SQG patches are also called {\it sharp fronts} as they model sharp temperature fonts in atmospheric and oceanic flows. Unlike general weak solutions, sharp fronts are known to be  unique. Precisely, \cite{Rodrigo} proved  local-in-time existence of $C^\infty$ periodic sharp fronts using a Nash-Moser inverse function theorem. Here, the regularity is that of  $\p D(t)$. The first  local-in-time existence result for finite regularity \cite{G} was obtained in $H^k$, $k\ge 3$.  Uniqueness was subsequently established in \cite{CCG} for  sharp front weak solutions with {\it bounded curvature}; namely, $\p D(t) \in C^{2+\delta}$. Recently,  \cite{GP} proved the existence of $H^2$ patches for gSQG with $\alpha\in (0, 1)$, the proof of which breaks down when $\alpha=1$ (SQG) and no uniqueness was given for $\alpha\in[\mez,1)$. Note that for  $H^2$ fronts, the  curvature  is merely $L^2$.  In view of the aforementioned results, a natural question is whether SQG sharp fronts with {\it unbounded curvature} are well-posed (existence and uniqueness). In the present paper, we affirmatively answer this by establishing local well-posedness of SQG sharp fronts in $H^{2+s}$ with arbitrarily small $s>0$.  

\begin{theo}\label{SQGTheorem}
	Let $s\in (0, \mez)$ be arbitrary. For any initial bounded domain $D(0)$ with non self-intersecting $H^{2+s}$ boundary, there exist $T>0$ and a unique solution $\tt$ of \eqref{transport}-\eqref{SQG}-\eqref{front} on $[0, T]$ such that the boundary of $D(t)$ remains non self-intersecting and belongs to $H^{2+s}$. 
\end{theo}
The importance of curvature in the study of sharp temperature fronts is supported by numerical analysis and theoretical results. In \cite{CFMR}, numerical evidence of curvature blow-up of the front boundary is shown when two distinct particles on the boundary collide at the same point. Later on, \cite{GS} proved that uniform control of the curvature prevents these particle collision type singularities. Further numerics \cite{ScottDritschel} indicate curvature blow-up for different singular solutions, specifically a cascade of filament instabilities with width tending to zero in finite time. On the other hand,  \cite{GP} proved that $H^3$ solutions can be continued provided that the contour is non self-intersecting and its $W^{2, p}$ norm is controlled for some $p>2$. 

Recently there has been a lot of effort to obtain different type of SQG front solutions. \textit{Almost sharp fronts} were given in \cite{FR,FR2} establishing high order asymptotics of smoothed out sharp fronts to exact sharp fronts. \cite{CCG-S} established the first nontrivial global-in-time SQG fronts which are  uniformly rotating smooth solutions. Surprisingly there also  exist a family of nontrivial analytic stationary fronts \cite{G-S}. In the case of infinite near planar fronts, recent works \cite{HSZ4,HSZ2} gave global-in-time solutions of small slope. Existence and construction of co-rotating and traveling pairs for the SQG equations were recently given in \cite{CQZZ,CQZZ2}. The main result in this paper provides the first construction of general  fronts with low regularity, allowing unbounded curvature.

\subsection{Contour dynamics equation  and reparametrization} 
Computing the normal velocity at the patch boundary, an evolution equation for the boundary can be obtained as
\begin{equation}\label{SQGold}
\p_tx(\g, t)=\int_{\T} \frac{\p_\g x(\g, t)-\p_\g x(\g-\eta, t)}{|x(\g, t)-x(\g-\eta, t)|}d\eta.
\end{equation}
See e.g. \cite{G} for a detailed derivation. The singularity in the denominator of \eqref{SQGold} can be compared to the difference in the parameter variable by $L^{\infty}$ control of the \textit{arc chord quantity},
\begin{equation}\label{arcchordquantity}
F(x)(\gamma,\eta, t) = \frac{|\eta|}{|x(\g, t)-x(\g-\eta, t)|}.
\end{equation}
This quantity blows up if the patch self-intersects or if the parametrization artificially causes the blow up. To rule out this second possibility, one can reformulate the evolution equation to choose a parametrization such that the modulus of the tangent vector to the boundary only depends on time, i.e. 
\bq\label{At}
|\dg x(\gamma,t)|^{2} = A(t).
\eq Furthermore, this choice of parameter will yield several useful techniques to deal with the low regularity of the sharp fronts. Adding a scalar multiple of the tangent $\dg x$ to \eqref{SQGold} does not change the shape of the patch boundary because it is independent of the normal velocity. Following the derivation in a previous paper \cite{G}, the scalar multiple can be computed explicitly. This yields the needed contour evolution equation satisfying the property \eqref{At}:
\bq\label{SQGp}
\p_tx(\g, t)=\int_{\T} \frac{\p_\g x(\g, t)-\p_\g x(\g-\eta, t)}{|x(\g, t)-x(\g-\eta, t)|}d\eta+\ld(\g, t)\p_\g x(\g, t),
\eq
\bq\label{ld}
\begin{aligned}
	\ld(\g, t)&=\frac{\pi +\g}{2\pi}\int_\T \int_\T \p_\g\Big(\frac{\p_\g x(\g)-\p_\g x(\g -\eta)}{|x(\g)-x(\g-\eta)|}\Big)\cdot \frac{\p_\g x(\g)}{|\p_\g x(\g)|^2}d\eta d\g\\
	&\quad-\int_{-\pi}^\g \int_\T \p_\g\Big(\frac{\p_\g x(\xi)-\p_\g x(\xi-\eta)}{|x(\xi)-x(\xi-\eta)|}\Big) \cdot \frac{\p_\g x(\xi)}{|\p_\g x(\xi)|^2}d\eta d\xi.
\end{aligned}
\eq
Using the preceding parametrization, Theorem \ref{SQGTheorem} can be restated as 
\begin{theo}\label{mainTheorem}
Let $s\in (0, \mez)$ be arbitrary. Let $x^o(\g)\in H^{2+s}(\T)$ such that $\p_\g x^o\cdot \p_\g^2 x^o=0$ and $F(x^o)\in L^\infty(\T\times \T)$. There exists $T>0$ depending only on $\| x^o\|_{H^{2+s}}$ and $\| F(x^o)\|_{L^\infty}$ such that \eqref{SQGp}-\eqref{ld} has a unique solution $x\in C([0, T]; H^{2+s}(\T))$ satisfying $F(x)\in L^\infty(\T\times \T\times [0, T])$ and $x\vert_{t=0}=x^o$. 
\end{theo}
\subsection{On the proof}
The right side of equation \eqref{SQGp} has a logarithmic singularity $1/|\eta|$. This singularity is a reason for the $H^2$ existence result in \cite{GP} only holding for gSQG with $\alpha<1$, in which case the denominator $|x(\g, t)-x(\g-\eta, t)|$ is replaced by $|x(\g, t)-x(\g-\eta, t)|^\alpha$. In order to reach  $\alpha=1$, our key idea is to work in $H^{2+s}$ and use the extra $s$ derivative to compensate the logarithmic singularity. The trade-off is that the control of the fractional Sobolev regularity $H^{2+s}$ induces various commutators. A priori estimates in $H^{2+s}$ are established in Sections \ref{sec:nontangentialestimates} and \ref{sec:tangentialestimates} using appropriate commutator estimates. These estimates in turn require in particular the control of $\ld$ in the H\"older and Sobolev spaces $C^{1+\eps}$ and $H^\tdm$, whose proof delicately uses the full $H^{2+s}$ regularity. See subsection \ref{subsec:ld}. A priori estimates for the arc chord quantity \eqref{arcchordquantity} also require new arguments and are obtained in Section \ref{sec:arcchord}. All the a priori estimates are gathered in Section \ref{sec:existence} to conclude the existence of solutions via regularizing initial data. The regularization procedure requires some care in order to guarantee the property \eqref{At} carries through the approximation scheme. 

The uniqueness part is subtle. Recall that the uniqueness result in \cite{CCG} for \eqref{SQGp}-\eqref{ld} was obtained via an $H^1$ contraction estimate. For unbounded  curvature this approach does not work neither for $L^2$ nor $H^2$ contraction estimate. To get a glimpse of  the difficulty, let us consider one of the more singular terms in the $H^{1}$ evolution of $z(\gamma,t) = x(\gamma,t)-y(\gamma,t)$, the difference of two solutions $x$ and $y$ in $H^{2+s}$:
\begin{equation*}
	I_{312}= \int \dg z(\gamma)\cdot \int (\dg^2 y(\gamma)-\dg^{2}y(\gamma-\eta))\Big(\frac{1}{|x(\gamma)- x(\gamma-\eta)|}-\frac{1}{|y(\gamma)- y(\gamma-\eta)|}\Big) d\e d\g.
\end{equation*}
In the regime of bounded curvature, guaranteed when $s\in (\mez, 1)$, we have that $\dg^{2}x,~\dg^{2}y \in C^{\delta}$ with $\delta=s-\mez$. We can then control 
$$I_{312} \lesssim \|\dg z\|_{L^{2}}^{2}\|\dg^2 y\|_{C^{\delta}}\|F(x)\|_{L^{\infty}_{\gamma,\eta}}\|F(y)\|_{L^{\infty}_{\gamma,\eta}}\int \frac{1}{|\eta|^{1-\delta}} d\e$$
which is integrable. Note that $C^2$ regularity would not suffice for this argument.  For $s\in (0, \mez)$, our idea is to prove contraction estimates not only for the $H^1$ norm but also the moduli of the tangential vectors $|\dg x(\gamma,t)| - |\dg y(\gamma,t)|$ to the patches. The entire argument has to avoid the use of the $C^2$ norm of the solution through detailed decompositions of terms. This is done in Section \ref{sec:uniqueness}.
\begin{rema}
Both a priori and contraction estimates crucially use the full $H^{2+s}$ regularity. Whether SQG sharp fronts are well-posed in $H^2$ remains an open problem. On the other hand, for higher regularity $H^{2+s}$ with $s\in [\mez, 1)$, local well-posedness should hold by modifications  of our proof for the existence part. 
\end{rema}

\section{ A  priori estimates for nontangential nonlinearity}\label{sec:nontangentialestimates}
	For notational simplicity we denote
\begin{align}\label{def:x-}
	&x_-(\gamma, \eta)=x(\gamma)-x(\gamma-\eta),\\ \label{def:g}
	& g=g(\g, \eta)=\frac{1}{|x_-(\g, \eta)|}.
\end{align}
When it is clear that $\g$ and $\eta$ are the variables,  we use the shorthand 
\bq
x_-=x_-(\g)=x_-(\gamma, \eta).
\eq
We shall also write for fixed $t$, 
\bq
\| F(x)\|_{L^\infty}=\| F(x)(\cdot, \cdot, t)\|_{L^\infty(\T\times \T)}.
\eq
We assume throughout this section  that $x\in C([0, T]; H^3)$ is a solution of \eqref{SQGp}-\eqref{ld}. 	Fixing $s\in (0, \mez)$, our goal is to establish  a priori estimates for $H^{2+s}$ norm of $x$. For high frequencies, we differentiate  \eqref{SQGp} twice in $\gamma$ to obtain
	\bq\label{eq:d2}
	\begin{aligned}
		\p_t \p_\g^2x&=\int_\T \p_\g^3x_- gd\eta+2\int_\T \p_\g^2x_-\p_\g gd\eta+\int_\T \p_\g x_-\p_\g^2 g d \eta
		+\ld \p_\g^3x+2\p_\g\ld \p_\g^2x+\p_\g^2\ld \p_\g x\\
		&:=\cN_1+\cN_2+\cN_3+\cT_1+\cT_2+\cT_3.
	\end{aligned}
	\eq
	We apply $\Ld^s$ to \eqref{eq:d2},  multiply the resulting equation by $\Ld^{s}\p_\g^2x$, then integrate over $\T$, leading to
	\bq\label{evol:Hs}
	\mez \frac{d}{dt}\| \Ld^s\p_\g^2 x\|_{L^2}^{2}=\sum_{j=1}^3 N_j+\sum_{j=1}^3 T_j.
	\eq
	We call $T_j$ tangential terms and $N_j$ nontangential terms. This section is devoted to a priori estimates for the nontangential terms. Tangential terms shall be considered in Section \ref{sec:tangentialestimates}.
	\subsection{Preliminaries}
	We shall frequently appeal to the identity 
	\bq\label{identity1}
	x_-\cdot \p_\g x_-=(x_--\eta \p_\g x(\g))\cdot \p_\g x_-+\mez \eta |\p_\g x_-|^2.
	\eq
	For the proof of \eqref{identity1} we use the fact that $|\p_\g x(\g, t)|^{2}=A(t)$ is independent of $\g$  to obtain
	\bq\label{dxdx-}
	\begin{aligned}
		\p_\g x(\gamma, t)\cdot \p_\g x_-(t)&=|\p_\g x(\gamma, t)|^2-\p_\g x(\gamma, t)\cdot \p_\g x(\gamma-\eta, t)\\
		&=\mez\Big(|\p_\g x(\gamma, t)|^2+|\p_\g x(\gamma-\eta, t)|^2-2\p_\g x(\gamma, t)\cdot \p_\g x(\gamma-\eta, t)\Big)=\mez  |\p_\g x_-|^2.
	\end{aligned}
	\eq
	\begin{lemm}
		For all $\eps\in [0, \mez)$,  we have
		\bq\label{bound:gL}
		| \p_\g g|\les \frac{1}{|\eta|^{1-2\eps}}\| F(x)\|_{L^\infty}^3\|\p_\g x\|_{C^{\mez+\eps}}^2
		\eq
		and
		\bq\label{bound:gC}
		\| \p_\g g(\cdot, \eta, t)\|_{\dot C^{s+\eps}}\les \frac{1}{|\eta|^{1-s+\eps}}\| F(x)\|_{L^\infty}^3\|\p_\g x\|^2_{C^{\mez+s}}+\frac{1}{|\eta|^{1-2s}}\| F(x)\|_{L^\infty}^5\| \p_\g x\|_{C^{\mez+s}}^2\| \p_\g x\|^2_{C^{\mez}}.
		\eq
	\end{lemm}
	\begin{proof}
		In view of \eqref{identity1} we have
		\bq\label{dg}
		\p_\g g=-\frac{(x_--\eta \p_\g x(\g))\cdot \p_\g x_-}{|x_-|^{3}}-\frac{1}{2}\frac{\eta |\p_\g x_-|^2}{|x_-|^{3}}:=- m_1(\g, \eta)-\frac{1}{2} m_2(\g, \eta).
		\eq
		By Taylor expansion,  it is readily seen that
		\[
		|(x_--\eta \p_\g x(\g))|\le |\eta|^{\tdm+\eps}\| \p_\g x\|_{C^{\mez+\eps}},\quad |\p_\g x_-|\le |\eta|^{\mez+\eps}\|\p_\g x\|_{C^{\mez+\eps}}.
		\]
		Consequently,
		\[
		| \p_\g g|\les \| F(x)\|_{L^\infty}^{3}\|\p_\g x\|_{C^{\mez+\eps}}^2\frac{|\eta|^{\tdm+\eps+\mez+\eps}+|\eta|^{1+2(\mez+\eps)}}{|\eta|^{3}},
		\]
		whence \eqref{bound:gL} follows.
		
		We turn to prove \eqref{bound:gC}. For $|h|\le \pi$ we consider the difference $\p_\g g(\g, \eta)-\p_\g g(\g+h, \eta)$. First, 
		\[
		\begin{aligned}
			m_2(\g, \eta)-m_2(\g+h, \eta)=& \eta \frac{|\p_\g x_-(\g, \eta)|^2}{|x_-(\g, \eta)|^{3}}-  \eta \frac{|\p_\g x_-(\g+h, \eta)|^2}{|x_-(\g+h, \eta)|^{3}}\\
			&=  \eta \frac{|\p_\g x_-(\g, \eta)|^2-|\p_\g x_-(\g+h, \eta)|^2}{|x_-(\g, \eta)|^{3}}\\
			&\quad+ \eta|\p_\g x_-(\g+h, \eta)|^2\Big(\frac{1}{|x_-(\g, \eta)|^{3}}-\frac{1}{|x_-(\g+h, \eta)|^{3}}\Big)\\
			&:=I+II.
		\end{aligned}
		\]
		We write
		\[
		\begin{aligned}
			&|\p_\g x_-(\g, \eta)|^2-|\p_\g x_-(\g+h, \eta)|^2\\
			&\quad=\Big[\p_\g x_-(\g, \eta)-\p_\g x_-(\g+h, \eta)\Big]\cdot \Big[\p_\g x_-(\g, \eta)+\p_\g x_-(\g+h, \eta) \Big]:=A\cdot B.
		\end{aligned}
		\]
		Clearly, 
		\[
		|B|\les |\eta|^{\mez+s}\| \p_\g x\|_{C^{\mez+s}}.
		\] 
		On one hand,
		\[
		|A|=|\eta\int_0^1[\p_\g^2 x(\g-r\eta)-\p_\g^2 x(\g+h-r\eta)]dr|\les |\eta| |h|^a\|\p_\g^2 x\|_{C^a}\quad \forall a\in (0, 1).
		\]
		On the other  hand, using the mean value theorem gives
		\[
		|A|=|\p_\g x_-(\g,-h)-\p_\g x_-(\g-\eta, -h)|\le |\p_\g x_-(\g,-h)|+|\p_\g x_-(\g-\eta, -h)|\les  |h|^b\| \p_\g x\|_{C^b}\quad\forall b\in (0, 1).
		\]  
		Therefore, regarding $A$ as a linear operator of $x$, we deduce by interpolation that
		\[
		|A| \les |h|^{\tt a+(1-\tt)b}|\eta|^\tt\|\p_\g x\|_{C^{\tt(1+a)+(1-\tt)b}}\quad \forall \tt\in (0, 1).
		\]
		Now for any $\eps\in [0, \mez)$ we choose $\tt=\mez-\eps$ and $a, b\in (0, 1)$ satisfying $\tt(1+a)+(1-\tt)b=s+\mez$. Then $ \tt a+(1-\tt)b=s+\mez-\tt=s+\eps$ and
		\bq\label{est:A}
		|A| \les |h|^{s+\eps}|\eta|^{\mez-\eps}\|\p_\g x\|_{C^{s+\mez}}.
		\eq
		Combining this with the above estimate for $B$ we get
		\bq\label{boundg:H:1}
		|I|\les \| F(x)\|_{L^\infty}^{3}\frac{|h|^{s+\eps}}{|\eta|^{1-s+\eps}}\|\p_\g x\|^2_{C^{\mez+s}}.
		\eq
		As for $II$ we first note that 
		\[
		|\p_\g x_-(\g+h, \eta)|^2\le |\eta|\| \p_\g x\|_{C^{\mez}}^2.
		\]
		Setting  $k(\g, \eta)=|x_-(\g, \eta)|^{-3}$ we have 
		\[
		\p_\g k(\g, \eta)=3 \frac{\p_\g g(\g, \eta)}{|x_-(\g)|^{2}}.
		\]
		Then using  the fundamental theorem of calculus and the bound  \eqref{bound:gL} for $\p_\g g$ we obtain 
		\bq\label{diff:k}
		\begin{aligned}
			&|k(\g +h, \eta)-k(\g, \eta)|
			&  \les  \frac{|h|}{|\eta|^2|\eta|^{1-2s}}\| F(x)\|_{L^\infty}^{5}\| \p_\g x\|_{C^{\mez+s}}^2.
		\end{aligned}
		\eq
		We have proved
		\bq\label{boundg:H:2}
		| II|\les \frac{|h|}{|\eta|^{1-2s}}\| F(x)\|_{L^\infty}^{5}\| \p_\g x\|_{C^{\mez+s}}^2\| \p_\g x\|^2_{C^{\mez}}
		\eq
		Next, we have
		\begin{align*}
			&m_1(\g, \eta)-m_1(\g+h, \eta)\\
			&=\frac{[x_-(\g, \eta)-\eta \p_\g x(\g)]\cdot \p_\g x_-(\g, \eta)}{|x_-(\g, \eta)|^3}-\frac{[x_-(\g+h, \eta)-\eta \p_\g x(\g+h)]\cdot \p_\g x_-(\g+h, \eta)}{|x_-(\g+h, \eta)|^3}\\\
			&=\frac{[x_-(\g, \eta)-\eta \p_\g x(\g)]\cdot \p_\g x_-(\g, \eta)-[x_-(\g+h, \eta)-\eta \p_\g x(\g+h)]\cdot \p_\g x_-(\g+h, \eta)}{|x_-(\g, \eta)|^3}\\
			&\quad+[x_-(\g+h, \eta)-\eta \p_\g x(\g+h)]\cdot \p_\g x_-(\g+h, \eta)\Big(\frac{1}{|x_-(\g, \eta)|^3}-\frac{1}{|x_-(\g+h, \eta)|^3}\Big)\\
			&=III+IV.
		\end{align*}
		Since 
		\[
		\left|[x_-(\g+h, \eta)-\eta \p_\g x(\g+h)]\cdot \p_\g x_-(\g+h, \eta)\right|\le |\eta|^2\| \p_\g x\|_{C^{\mez}}^2.
		\]
		Combining this and \eqref{diff:k} we deduce that $IV$ obeys the same bound as $II$:
		\bq\label{boundg:H:4}
		|IV| \les \frac{|h|}{|\eta|^{1-2s}}\| F(x)\|_{L^\infty}^5\| \p_\g x\|_{C^{\mez+s}}^2\| \p_\g x\|^2_{C^{\mez}}.
		\eq
		We write 
		\[
		\begin{aligned}
			III&=\frac{1}{|x_-(\g, \eta)|^3}
			\p_\g x_-(\g, \eta)\cdot [x_-(\g, \eta)-\eta \p_\g x(\g)-x_-(\g+h, \eta)+\eta \p_\g x(\g+h)]\\
			&\quad+\frac{1}{|x_-(\g, \eta)|^3}
			[\p_\g x_-(\g, \eta)-\p_\g x_-(\g+h, \eta)]\cdot [x_-(\g+h, \eta)-\eta \p_\g x(\g+h)]\\
			&=III_1+III_2.
		\end{aligned}
		\]
		Combining the estimate
		\[
		|x_-(\g+h, \eta)-\eta \p_\g x(\g+h)|\le|\eta|^{\tdm+s}\|\p_\g x\|_{C^{\mez+s}}
		\]
		and \eqref{est:A} we obtain
		\[
		|III_2|\les \| F(x)\|_{L^\infty}^3\frac{|h|^{s+\eps}}{|\eta|^{1-s+\eps}}\|\p_\g x\|^2_{C^{\mez+s}}
		\quad\forall \eps\in [0, \mez).
		\]
		By the mean value theorem,
		\begin{align*}
		C&:=x_-(\g, \eta)-\eta \p_\g x(\g)-x_-(\g+h, \eta)+\eta \p_\g x(\g+h)\\
		&=-\frac{\eta^2}{2}\int_0^1(1-r)[\p_\g^2 x(\g-r\eta)-\p_\g^2 x(\g+h-r\eta)]dr
		\end{align*}
		and thus 
		\[
		|C| \les \eta^2 |h|^a\|\p_\g^2 x\|_{C^a}\quad \forall a\in (0, 1).
		\]
		On the other hand, since 
		\[
		C=x_-(\g, -h)-x_-(\g-\eta, -h)-\eta\p_\g x_-(\g, -h)=\eta \int_0^1 \p_\g x_-(\g-r\eta, -h)dr-\eta\p_\g x_-(\g, -h)
		\]
		we get
		\[
		|C|\les |\eta||h|^b\| \p_\g x\|_{C^b}\quad\forall b\in (0, 1).
		\]
		By an interpolation argument as for \eqref{est:A} we obtain 
		\[
		|C| \les |h|^{s+\eps}|\eta|^{\tdm-\eps}\|\p_\g x\|_{C^{s+\mez}}\quad\forall \eps\in [0, \mez).
		\]
		Consequently, 
		\[
		\begin{aligned}
			|III_1|&\les \| F(x)\|_{L^\infty}^3\frac{|\eta|^{\mez+s}}{|\eta|^3}\|\p_\g x\|_{C^{s+\mez}}|h|^{s+\eps}|\eta|^{\tdm-\eps}\|\p_\g x\|_{C^{s+\mez}}\\
			&\les \| F(x)\|_{L^\infty}^3\frac{|h|^{s+\eps}}{|\eta|^{1-s+\eps}}\|\p_\g x\|^2_{C^{\mez+s}}\quad\forall \eps\in [0, \mez).
		\end{aligned}
		\]
		We have proved
		\bq\label{boundg:H:3}
		|III| \les \| F(x)\|_{L^\infty}^3\frac{|h|^{s+\eps}}{|\eta|^{1-s+\eps}}\|\p_\g x\|^2_{C^{\mez+s}}\quad\forall \eps\in [0, \mez).
		\eq
		Putting together \eqref{boundg:H:1}, \eqref{boundg:H:2}, \eqref{boundg:H:4} and \eqref{boundg:H:3}  we conclude the proof of \eqref{bound:gC}.
	\end{proof}
	\begin{lemm}\label{lemm:cmt}
		For all $r\in (0, 1)$ and $\delta>0$, there exists $C>0$ such that
		\bq\label{cmt:1}
		\| [\Ld^r, u]v\|_{L^2}\le C \| u\|_{B^{1+\delta}_{\infty, \infty}}\| v\|_{H^{r-1}}
		\eq
		provided that the right-hand side is finite. 
	\end{lemm}
	\begin{proof}
		Using the Bony decomposition $uv=T_uv+T_vu+R(u, v)$  we write
		\[
		[\Ld^r, u]v= [\Ld^r, T_u]v+\Ld^r(T_vu)+\Ld^rR(u, v)-T_{\Ld^r v}u-R(\Ld^rv, u).
		\]
		It follows from paraproduct rules that 
		\begin{align*}
			\| [\Ld^r, T_u]v\|_{L^2}&\les \|  \na u\|_{L^\infty}\| v\|_{H^{r-1}}\quad \eqref{cmt:pp},\\
			\| \Ld^r(T_vu)\|_{L^2}&\les \| T_vu\|_{ H^r}\les \| u\|_{B^{1}_{\infty, 2}}\| v\|_{ H^{r-1}}\quad\eqref{pp:new1},\\
			\| \Ld^rR(u, v)\|_{L^2}&\les \| R(u, v)\|_{ H^r}\les \| u\|_{B^{1}_{\infty, \infty}}\| v\|_{ H^{r-1}}\quad\eqref{BonyR},\\
			\| T_{\Ld^r v}u\|_{L^2}&\les \| u\|_{B^{1}_{\infty, 2}}\| \Ld^r v\|_{ H^{-1}}\les \| u\|_{B^{1}_{\infty, 2}}\| v\|_{ H^{r-1}}\quad\eqref{pp:new1},\\
			\| R(\Ld^rv, u)\|_{L^2}&\les \|  u\|_{B^{1+\delta}_{\infty, \infty}}\| \Ld^rv\|_{H^{-1}}\les  \| u\|_{B^{1+\delta}_{\infty, \infty}}\| v\|_{H^{r-1}}\quad\eqref{BonyR}.
		\end{align*}
		Since $B^{1+\delta}_{\infty, \infty}\subset B^{1}_{\infty, 2}$, the proof is complete.
	\end{proof}
	\subsection{Control of $N_1$}
	By symmetrizing and integration by parts, we have
	\[
	\begin{aligned}
		N_1=\int_\T \Ld_\g^s\p_\g^2x(\g) \cN_1(\g)d\g&=\mez\int_\T\int_\T \Ld_\g^{2s}\p_\g^2x_- \p_\g^3 x_-gd\eta d\g\\
		&=\mez\int_{\T}\int_\T  \Ld^s_\g \p_\g^2 x_-\Ld^s_\g \p_\g^3 x_-  g  d\g d\eta+\mez\int_{\T}\int_\T \Ld^s_\g \p_\g^2x_-[\Ld^s_\g, g]\p_\g^3 x_- d\g d\eta\\
		&=-\frac{1}{4}\int_{\T}\int_\T |\Ld^s_\g \p_\g^2 x_-|^2\p_\g gd\g d\eta +\mez\int_{\T}\int_\T \Ld^s_\g \p_\g^2x_-[\Ld^s_\g, g]\p_\g^3 x_- d\g d\eta\\
		&=N_{11}+N_{12}.
	\end{aligned}
	\]
	By virtue of the Lipschitz bound \eqref{bound:gL} for $g$, we have
	\begin{align*}
		|N_{11}|&\les \| F(x)\|_{L^\infty}^3\|\p_\g x\|_{C^{\mez+s}}^2\int_{\T}\int_\T |\Ld^s_\g \p_\g^2 x_-|^2d\g \frac{d\eta}{|\eta|^{1-2s}}\\
		&\les \| F(x)\|_{L^\infty}^3\|\p_\g x\|_{C^{\mez+s}}^2\| \Ld^s\p_\g^2x\|_{L^2}^2.
	\end{align*}
	Regarding $N_{12}$ we appeal to the commutator estimate \eqref{cmt:1} and notice that  $[\Ld^s, g]=[\Ld^s, g-\wh{g}(0)]$, where $g=g(\cdot, \eta)$. Then for any $\delta>0$,
	\[
	\| [\Ld^s_\g, g]\p_\g^3 x_- \|_{L^2}\les \|  g-\wh{g}(0)\|_{B^{1+\delta}_{\infty, \infty}}\| \p_\g^2 x_-\|_{H^s}\les  \|  \p_\g g\|_{\dot B^{\delta}_{\infty, \infty}}\| \p_\g^2 x_-\|_{\dot H^s}\les   \|  \p_\g g\|_{\dot C^\delta}\| \p_\g^2 x_-\|_{\dot H^s},
	\]
	where we have used Proposition \ref{equi:HZ} in the last inequality. Now invoking the (homogeneous) H\"older estimate \eqref{bound:gC} for $\p_\g g$ with  $\eps=0$ yields
	\begin{align*}
		|N_{12}|&\les \Big(\| F(x)\|_{L^\infty}^3\| \p_\g x\|_{C^{\mez+s}}^2+\| F(x)\|_{L^\infty}^5\| \p_\g x\|_{C^{\mez+s}}^4\Big)\int_{\T}\int_\T |\Ld^s_\g \p_\g^2x_-| \|\p_\g^3 x_-\|_{\dot H^{s-1}_\g} d\g  \frac{d\eta}{|\eta|^{1-s}} \\
		&\les   \Big(\| F(x)\|_{L^\infty}^3\| \p_\g x\|_{C^{\mez+s}}^2+\| F(x)\|_{L^\infty}^5\| \p_\g x\|_{C^{\mez+s}}^4\Big)\| \Ld^s\p_\g^2x\|^2_{L^2}.
	\end{align*}
	In view of this and the above estimate for $N_{11}$ we have proved
	\bq\label{est:N1}
	|N_1|\les   \Big(\| F(x)\|_{L^\infty}^3\| \p_\g x\|_{C^{\mez+s}}^2+\| F(x)\|_{L^\infty}^5\| \p_\g x\|_{C^{\mez+s}}^4\Big)\| \Ld^s\p_\g^2x\|^2_{L^2}.
	\eq
	\subsection{Control of $N_2$}
	\begin{align*}
		N_2=\int_\T \Ld^{2s}_\g\p_\g^2x(\g)\cdot \cN_2(\g)d\g &
		=2\int_\T \Ld^{s}_\g\p_\g^2x(\g) \cdot\int_\T  \Ld^s(\p_\g^2x_-\p_\g g)d\eta \\
		&\le 2\| \Ld^{s}_\g\p_\g^2x\|_{L^2}\int_\T\| \Ld^s(\p_\g^2x_-\p_\g g)\|_{L^2_\g}d\eta.
	\end{align*}
	Combining the paraproduct rules \eqref{BonyR}, \eqref{pp:1} and \eqref{pp:new1}, we get
	\bq\label{productest:1}
	\begin{aligned}
		\|\Ld^s(u v)\|_{L^2}\le \|u v\|_{H^s}&\les \|v\|_{L^\infty}\|u\|_{H^s}+ \| u\|_{L^2}\|v\|_{B^{s+\delta/2}_{\infty, 2}}\\
		&\les \ \|v\|_{L^\infty}\|u\|_{H^s}+\| u\|_{L^2}\|v\|_{C^{s+\delta}}
	\end{aligned}
	\eq
	for any $\delta\in (0, \mez)$. Note that we have used Proposition \ref{equi:HZ} to have $\|v\|_{B^{s+\delta}_{\infty, \infty}}\simeq \|v\|_{C^{s+\delta}}$. We apply \eqref{productest:1} with $u=\p_\g^2x_-(\cdot, \eta)$ and $v=\p_\g g(\cdot, \eta)$. Using the $L^\infty$ estimate \eqref{bound:gL} and the H\"older estimate \eqref{bound:gC}, we deduce 
	\begin{align*}
		\|\Ld^s(\p_\g^2x_-\p_\g g)\|_{L^2_\g}&\les \| \Ld^s\p_\g^2x\|_{L^2}\| F(x)\|_{L^\infty}^3\|\p_\g x\|_{C^{\mez+s}}^2\frac{1}{|\eta|^{1-2s}}\\
		&\quad+ \| \p_\g^2x\|_{L^2}\Big(\| F(x)\|_{L^\infty}^3\| \p_\g x\|_{C^{\mez+s}}^2+\| F(x)\|_{L^\infty}^5\| \p_\g x\|_{C^{\mez+s}}^4\Big)\frac{1}{|\eta|^{1-s}}.
	\end{align*}
	Therefore, 
	\bq\label{est:N2}
	|N_2|\les \Big(\| F(x)\|_{L^\infty}^3\| \p_\g x\|_{C^{\mez+s}}^2+\| F(x)\|_{L^\infty}^5\| \p_\g x\|_{C^{\mez+s}}^4\Big) \| \Ld^s\p_\g^2x\|_{L^2}^2.
	\eq
	\subsection{Control of $N_3$}
	Clearly,
	\[
	|N_3|\le \| \Ld^s\p_\g^2x\|_{L^2}\int_\T \| \Ld^s(\p_\g  x_-\p_\g^2g)\|_{L^2_\g}d\eta.
	\]
	We compute 
	\[
	\p_\g  x_-\p_\g^2g=3 \p_\g x_-\frac{(x_-\cdot\p_\g x_-)^2}{|x_-|^5}-\alpha  \p_\g x_-\frac{|\p_\g x_-|^2+x_-\cdot \p_\g^2x_-}{|x_-|^3}.
	\]
	We shall only consider the term $n=\p_\g x_-\frac{x_-\cdot \p_\g^2x_-}{|x_-|^3}$ because the other terms can be treated similarly. Using $\p_\g x\cdot \p_\g^2x= 0$ we write 
	\[
	n=\frac{\p_\g x_-}{|x_-|^3}x_-\cdot \p_\g^2x_-=\frac{\p_\g x_-}{|x_-|^3}(x_--\eta\p_\g x)\cdot \p_\g^2x_--\frac{\p_\g x_-}{|x_-|^3}\eta \p_\g x_-\cdot\p_\g^2 x(\g-\eta):=n_1+n_2.
	\]
	By virtue of inequality \eqref{productest:1} we have
	\[
	\| \Ld^s n_1\|_{L^2_\g}\les  \|\Ld^s\p_\g^2x_-\|_{L^2_\g} \left\| \frac{\p_\g x_-}{|x_-|^3}(x_--\eta\p_\g x)\right\|_{L^\infty_\g}+  \|\p_\g^2x_-\|_{L^2_\g}\left\| \frac{\p_\g x_-}{|x_-|^3}(x_--\eta\p_\g x)\right\|_{C^{s+\eps}}
	\]
	for all $\eps\in (0, \mez)$.  By the mean value theorem,
	\begin{align*}
		\left|\frac{\p_\g x_-}{|x_-|^3}(x_--\eta\p_\g x)\right|&\les \frac{|\eta|^{\mez+s}\|  \p_\g x\|_{C^{\mez+s}} }{|\eta|^3}\| F(x)\|^3_{L^\infty}|\eta|^{\tdm+s}\|  \p_\g x\|_{C^{\mez+s}}\\
		&\les \frac{1}{|\eta|^{1-2s}}\| F(x)\|^3_{L^\infty}\|  \p_\g x\|_{C^{\mez+s}}^2.
	\end{align*}
	On the other hand, as for the term $m_1$ in \eqref{dg} we have (see \eqref{boundg:H:4} and \eqref{boundg:H:3}) 
	\begin{align*}
		\left\|\frac{\p_\g x_-}{|x_-|^3}(x_--\eta\p_\g x)\right\|_{C^{s+\eps}_\g}&\les \frac{1}{|\eta|^{1-s+\eps}}\| F(x)\|_{L^\infty}^{3}\|\p_\g x\|^2_{C^{\mez+s}}\\
		&\quad+\frac{1}{|\eta|^{1-2s}}\| F(x)\|_{L^\infty}^{5}\| \p_\g x\|_{C^{\mez+s}}^2\| \p_\g x\|^2_{C^{\mez}}\quad\forall \eps\in [0, \mez).
	\end{align*}
	Thus choosing $\eps=\frac{s}{2}$ we arrive at
	\begin{align*}
		\| \Ld^sn_1\|_{L^2_\g}&\les \frac{1}{|\eta|^{1-\frac{s}{2}}}\| \Ld^s\p_\g^2x\|_{L^2}\| F(x)\|_{L^\infty}^3\|\p_\g x\|_{C^{\mez+s}}^2\\
		&\quad+ \| \p_\g^2x\|_{L^2}\Big(\frac{1}{|\eta|^{1-\frac{s}{2}}}\| F(x)\|_{L^\infty}^3\| \p_\g x\|_{C^{\mez+s}}^2+\frac{1}{|\eta|^{1-2s}}\| F(x)\|_{L^\infty}^5\| \p_\g x\|_{C^{\mez+s}}^4\Big).
	\end{align*}
	By a similar argument it can be shown that $n_2$ obeys the same bound, and so does  $\p_\g  x_-\p_\g^2g$. We conclude that 
	\bq\label{est:N3}
	|N_3|\les \Big(\| F(x)\|_{L^\infty}^3\| \p_\g x\|_{C^{\mez+s}}^2+\| F(x)\|_{L^\infty}^5\| \p_\g x\|_{C^{\mez+s}}^4\Big) \| \Ld^s\p_\g^2x\|_{L^2}^2.
	\eq
	Putting together \eqref{est:N1}, \eqref{est:N2} and \eqref{est:N3}, we obtain
	\begin{prop}\label{prop:Nj}
	For all $s\in (0, \mez)$, there exists a polynomial $\cP(\cdot, \cdot)$ such that 
	\bq\label{est:Nj}
	\sum_{j=1}^3|N_j|\le \cP(\| F(x)\|_{L^\infty}, \| \p_\g x\|_{C^{\mez+s}})\| \p_\g^2x\|_{H^s}^2.
	\eq
	\end{prop}
	\section{ A  priori estimates for tangential nonlinearity}\label{sec:tangentialestimates}
	In this section, we derive a priori estimates for the tangential terms $T_j$ in \eqref{evol:Hs} assuming that  $x\in C([0, T]; H^3)$ is a solution of \eqref{SQGp}-\eqref{ld}. This requires good bounds for $\ld$. 
	\subsection{Estimates for $\lambda$}\label{subsec:ld}
	We first recall that 
	\[
	\begin{aligned}
		\ld(\g, t)&=\frac{\pi +\g}{2\pi}\int_\T \int_\T \p_\g\Big(\frac{\p_\g x(\g)-\p_\g x(\g -\eta)}{|x(\g)-x(\g-\eta)|}\Big)\cdot \frac{\p_\g x(\g)}{|\p_\g x(\g)|^2}d\eta d\g\\
		&\quad-\int_{-\pi}^\g \int_\T \p_\g\Big(\frac{\p_\g x(\xi)-\p_\g x(\xi-\eta)}{|x(\xi)-x(\xi-\eta)|}\Big) \cdot \frac{\p_\g x(\xi)}{|\p_\g x(\xi)|^2}d\eta d\xi.
	\end{aligned}
	\]
	Taking derivative with respect to $\g$ and recalling that $A(t)=|\p_\g x(\g)|^2$ depends only on $t$, we obtain
	\bq
\begin{aligned}
		\p_\g\ld(\g, t)&=\frac{1}{2\pi}\int_\T \int_\T \p_\g\Big(\frac{\p_\g x(\g)-\p_\g x(\g -\eta)}{|x(\g)-x(\g-\eta)|}\Big)\cdot \frac{\p_\g x(\g)}{|\p_\g x(\g)|^2}d\eta d\g\\
		&\quad+\frac{1}{A(t)}\int_\T \frac{x_-\cdot \p_\g x_-(\g)}{|x_-(\g)|^3}\p_\g x_-(\g)\cdot \p_\g x(\g)d\eta	 -\frac{1}{A(t)}\int_\T \frac{\p_\g^2 x_-(\g)}{|x_-(\g)|}\cdot \p_\g x(\g)d\eta,
	\end{aligned}
	\eq
	where $x_-(\g)=x(\g)-x(\g-\eta)$.  Using the identities $ \dg x(\gamma) \cdot \dg^{2}x_{-}(\gamma)=-\dg x_{-}(\gamma) \cdot \dg^{2}x(\gamma-\eta)$ and $\dg x(\gamma) \cdot \dg x_{-}(\gamma) = \frac{1}{2}|\dg x_{-}(\gamma)|^{2}$, we rewrite 
\begin{equation}\label{dgammalambda}
	\partial_{\gamma}\lambda(\gamma) = \Gamma_{1}(\gamma) +  \Gamma_{2}(\gamma) +  \Gamma_{3}(\gamma)
\end{equation}
where
\begin{align}\label{dglambda1}
	& \Gamma_{1}(\gamma) =  \frac{1}{A(t)} \int_{\mathbb{T}} \frac{\partial_{\gamma}x_{-}(\gamma)}{|x_{-}(\gamma)|}\cdot\dg^{2}x(\g-\eta) d\eta,
\end{align}
\begin{align}\label{dglambda2}
	& \Gamma_{2}(\gamma) =   \frac{1}{2A(t)}\int_{\mathbb{T}} \frac{|\partial_{\gamma}x_{-}(\gamma)|^{2} x_{-}(\gamma)\cdot\partial_{\gamma}x_{-}(\gamma)}{|x_{-}(\gamma)|^{3}} d\eta,
\end{align}
\begin{align}\label{dglambda3}
	& \Gamma_{3} =  \frac{-1}{2\pi A(t)}\int_{\mathbb{T}}\int_{\mathbb{T}} \frac{\partial_{\gamma}x_{-}(\gamma)}{|x_{-}(\gamma)|}\cdot\dg^{2}x(\g) d\eta d\gamma.
\end{align}
Note that $ \Gamma_{3}$ is a constant in $\g$. 
	\subsubsection{H\"older estimate for $\p_\g \ld$}
	We start with an  $L^\infty$ bound for $\p_\g \ld$.
	\begin{lemm}
	For all $\mu\in (0, \mez)$, there exists $C>0$ such that
	\bq\label{dld:Linfty}
	\|  \ld\|_{W^{1, \infty}}\le C\big(\| F(x)\|_{L^\infty}^3+\| F(x)\|_{L^\infty}^4\| \p_\g x\|_{L^\infty}\big)\| \p_\g x\|_{C^{\mez+\mu}}\| \p_\g^2x\|_{L^2}.
	\eq
	\end{lemm}
	\begin{proof}
	It suffices to prove that $\| \p_\g\ld \|_{L^\infty}$ is controlled by the right-hand side of \eqref{dld:Linfty}. Indeed, since $\ld(-\pi)=\ld(\pi)=0$, the Poincare inequality $\| \ld\|_{L^\infty}\le 2\pi\| \p_\g \ld\|_{L^\infty}$ holds. 
	
	From the definition of $F$ we have 
\[
\frac{1}{A(t)}\le \| F(x(t))\|_{L^\infty}^2.
\]
Hence, $\Gamma_1$ is bounded as
\begin{align*}
|\Gamma_1|&\le \| F(x)\|_{L^\infty}^3\| \p_\g x\|_{C^{\mez+\mu}}\int_\T \frac{|\p_\g^2x(\g-\eta)|}{|\eta|^{\mez-\mu}}d\eta\les \| F(x)\|_{L^\infty}^3\| \p_\g x\|_{C^{\mez+\mu}}\| \p_\g^2x\|_{L^2}.
\end{align*}
On the other hand,   
\begin{align*}
|\Gamma_2|&\le  \| F(x)\|_{L^\infty}^2\| \p_\g x\|_{L^\infty}\int_{\mathbb{T}}  \| \p_\g x\|_{C^{\mez+\mu}}|\eta|^{\mez+\mu}\frac{\| F(x)\|_{L^\infty}^2}{|\eta|^2} |\eta|\int_0^1|\p_\g^2 x(\g-r\eta)|drd\eta\\
&\les  \| F(x)\|_{L^\infty}^4 \| \p_\g x\|_{L^\infty} \| \p_\g x\|_{C^{\mez+\mu}}\int_0^1\frac{dr}{r^{\mez}} \Big(\int_{\mathbb{T}}\frac{1}{|\eta|^{1-2\mu}}d\eta\Big)^{\mez}\| \p_\g^2 x\|_{L^2}
\\&\les \| F(x)\|_{L^\infty}^4 \| \p_\g x\|_{L^\infty} \| \p_\g x\|_{C^{\mez+\mu}}\| \p_\g^2 x\|_{L^2},
\end{align*}
and
\begin{align*}
|\Gamma_3|&\le\| F(x)\|_{L^\infty}^2\int_{\mathbb{T}}\int_{\mathbb{T}} \| F(x)\|_{L^\infty}\| \p_\g x\|_{C^\mu}\frac{d\eta}{|\eta|^{1-\mu}} |\dg^{2}x(\g)|d\gamma\\
&\les \| F(x)\|_{L^\infty}^3\| \p_\g x\|_{C^\mu} \| \p^2_\g x\|_{L^2}.
 \end{align*}
	\end{proof}
	Next, we prove the following H\"older estimate for $\p_\g\ld$.
	\begin{prop}\label{prop:ld:Holder}
	For any $\mu \in (0, \mez)$, if $a\in (0, 1)$ satisfies 
\bq
a(\mez+\mu)<\mu
\eq
 then there exists a polynomial $\cP(\cdot, \cdot)$ independent of $x$ such that
\begin{equation}\label{lambdaCepsilon}
	\|\dg \lambda\|_{C^{a(\mez+\mu)}} \le \cP(\| F(x)\|_{L^\infty},\|\dg x\|_{C^{\frac{1}{2}+\mu}}) \|\dg^{2}x\|_{H^{\mu}}.
\end{equation} 
\end{prop}
\begin{proof}
We shall prove that the finite difference $\partial_{\gamma}\lambda(\gamma+h) - \dg\lambda(\gamma) = \sum_{i=1}^{2}  \Gamma_{i}(\gamma+h) -  \Gamma_{i}(\gamma)$ can be bounded by a positive power of $h$. The highest order term is
\begin{align*}
	 \Gamma_{1}(\gamma+h) -  \Gamma_{1}(\gamma)&=  \Gamma_{11} +  \Gamma_{12} +  \Gamma_{13}
\end{align*}
where
$$
	\Gamma_{11} = \frac{1}{A(t)}\int_{\mathbb{T}} \frac{\partial_{\gamma}x_{-}(\gamma)}{|x_{-}(\gamma)|}\cdot (\dg^{2}x(\g-\eta+h)-\dg^{2}x(\g-\eta)) d\eta,
$$
$$
\Gamma_{12} =\frac{-1}{A(t)} \int_{\mathbb{T}} \partial_{\gamma}x_{-}(\gamma)\cdot \dg^{2}x(\g-\eta+h) \Big(\frac{1}{|x_{-}(\gamma)|} - \frac{1}{|x_{-}(\gamma + h)|}\Big)d\eta,
$$
$$
\Gamma_{13} =\frac{-1}{A(t)} \int_{\mathbb{T}} (\partial_{\gamma}x_{-}(\gamma)-\partial_{\gamma}x_{-}(\gamma+h))\cdot \frac{\dg^{2}x(\g-\eta+h)}{ |x_{-}(\gamma + h)|}d\eta.
$$
On one hand, by Cauchy-Schwarz's inequality 
\begin{equation}\label{interld:est1}
	\begin{aligned}
		| \Gamma_{11}| &\leq \| F(x)\|_{L^\infty}^2\int_{\mathbb{T}} \frac{\|\dg x\|_{C^{\mez+\mu}}\|F(x)\|_{L^{\infty}}}{|\eta|^{\mez-\mu}}|\dg^{2}x(\g-\eta+h)-\dg^{2}x(\g-\eta)| d\eta\\
		&\leq C\|F(x)\|^3_{L^{\infty}} \|\dg x\|_{C^{\mez+\mu}} \|\dg^{2}x\|_{L^{2}}.
	\end{aligned}
\end{equation}
On the other hand, using the mean value theorem before applying  Cauchy-Schwarz's inequality, we have
\begin{equation}\label{interld:est2}
	\begin{aligned}
		| \Gamma_{11}| &\leq \| F(x)\|_{L^\infty}^2\int_{\mathbb{T}} \frac{\|\dg x\|_{C^{\mez+\mu}}\|F(x)\|_{L^{\infty}}}{|\eta|^{\mez-\mu}}\int_{0}^{1}|h||\dg^{3}x(\g-\eta+rh)| dr d\eta  \\
		&\leq  C \|F(x)\|^3_{L^{\infty}}|h| \|\dg x\|_{C^{\mez+\mu}} \|\dg^{3}x\|_{L^{2}}.
	\end{aligned}
\end{equation}
Viewing $ \Gamma_{11}$ as a linear operator in $\p_\g^2$ and interpolating between \eqref{interld:est1} and \eqref{interld:est2} gives
$$| \Gamma_{11}| \leq C\|F(x)\|^3_{L^{\infty}} |h|^{\mu} \|\dg x\|_{C^{\mez+\mu}}\|\dg^{2}x\|_{H^{\mu}}.$$

Next, for $ \Gamma_{12}$, we have
\begin{align*}
	| \Gamma_{12}| &\leq \| F(x)\|_{L^\infty}^2\int_{\mathbb{T}} |\partial_{\gamma}x_{-}(\gamma)||\dg^{2}x(\g-\eta+h)| \frac{|x_{-}(\gamma + h)-x_{-}(\gamma)|}{|x_{-}(\gamma + h)||x_{-}(\gamma)|}d\eta\\
	&\lesssim \| F(x)\|_{L^\infty}^4\int_{\mathbb{T}}\int_{0}^{1} \|\partial_{\gamma}x\|_{C^{\frac{1}{2}+\mu}}|\dg^{2}x(\g-\eta+h)| \frac{|h||\dg x_{-}(\gamma + rh)|}{|\eta|^{\frac{3}{2}-\mu}} dr d\eta\\
	&\lesssim \| F(x)\|_{L^\infty}^4\|\partial_{\gamma}x\|_{C^{\frac{1}{2}+\mu}} \int_{\mathbb{T}}\int_{0}^{1}\int_{0}^{1} |\dg^{2}x(\g-\eta+h)| \frac{|h||\dg^{2}x(\gamma+r'\eta+rh)|}{|\eta|^{\frac{1}{2}-\mu}} dr dr' d\eta,
\end{align*}
and therefore
\begin{align*}
	| \Gamma_{12}|&\lesssim \| F(x)\|_{L^\infty}^4\|\partial_{\gamma}x\|_{C^{\frac{1}{2}+\mu}}\|\dg^2 x\|_{L^{2}_{r}}|h|^{\frac{1}{2}} \int_{\mathbb{T}} \frac{|\dg^{2}x(\g-\eta+h)|}{|\eta|^{\frac{1}{2}-\mu}} d\eta\\
	&\lesssim \| F(x)\|_{L^\infty}^4|h|^{\frac{1}{2}}\|\partial_{\gamma}x\|_{C^{\frac{1}{2}+\mu}}\|\dg^2 x\|_{L^{2}}^{2}.
\end{align*}
For $ \Gamma_{13}$, choosing  $a\in (0, 1)$ such that $a(\mez+\mu)<\mu$, we estimate
\begin{align*}
	| \Gamma_{13}| &\leq  \| F(x)\|_{L^\infty}^2\int_{\mathbb{T}} |\partial_{\gamma}x_{-}(\gamma)-\partial_{\gamma}x_{-}(\gamma+h)|\frac{|\dg^{2}x(\g-\eta+h)|}{ |\eta|}d\eta\\
	&=  \| F(x)\|_{L^\infty}^2\int_{\mathbb{T}}  |\partial_{\gamma}x_{-}(\gamma)-\partial_{\gamma}x_{-}(\gamma+h)|^{a}|\partial_{\gamma}x_{-}(\gamma)-\partial_{\gamma}x_{-}(\gamma+h)|^{1-a}\frac{|\dg^{2}x(\g-\eta+h)|}{ |\eta|}d\eta\\
	&\lesssim  \| F(x)\|_{L^\infty}^2\|\partial_{\gamma}x\|_{C^{\frac{1}{2}+\mu}}\int_{\mathbb{T}}  |h|^{a(\frac{1}{2}+\mu)}\frac{|\dg^{2}x(\g-\eta+h)|}{|\eta|^{1-(1-a)(\frac{1}{2}+\mu)}}d\eta\\
	&\lesssim  \| F(x)\|_{L^\infty}^2|h|^{a(\frac{1}{2}+\mu)}\|\partial_{\gamma}x\|_{C^{\frac{1}{2}+\mu}}\|\p_\g^2x\|_{L^2}.
\end{align*}
The finite difference for $ \Gamma_2$ is decomposed as
\begin{equation}
	 \Gamma_{2}(\gamma+h)-  \Gamma_{2}(\gamma) =  \Gamma_{21} +  \Gamma_{22} +  \Gamma_{23} +  \Gamma_{24} 
\end{equation}
where
$$\Gamma_{21} = \frac{1}{2A(t)}\int_{\mathbb{T}} \frac{(|\partial_{\gamma}x_{-}(\gamma+h)|^{2}-|\partial_{\gamma}x_{-}(\gamma)|^{2}) x_{-}(\gamma)\cdot\partial_{\gamma}x_{-}(\gamma)}{|x_{-}(\gamma)|^{3}} d\eta,
$$
$$\Gamma_{22} = \frac{1}{2A(t)}\int_{\mathbb{T}} \frac{|\partial_{\gamma}x_{-}(\gamma+h)|^{2} (x_{-}(\gamma+h)- x_{-}(\gamma))\cdot\partial_{\gamma}x_{-}(\gamma)}{|x_{-}(\gamma)|^{3}} d\eta,$$
$$\Gamma_{23} = \frac{1}{2A(t)}\int_{\mathbb{T}} \frac{|\partial_{\gamma}x_{-}(\gamma+h)|^{2}  x_{-}(\gamma+h)\cdot(\partial_{\gamma}x_{-}(\gamma+h)-\partial_{\gamma}x_{-}(\gamma))}{|x_{-}(\gamma)|^{3}} d\eta,$$
and
$$\Gamma_{24} = \frac{1}{2A(t)}\int_{\mathbb{T}} |\partial_{\gamma}x_{-}(\gamma+h)|^{2}  x_{-}(\gamma+h)\cdot\partial_{\gamma}x_{-}(\gamma+h)\Big(\frac{1}{|x_{-}(\gamma+h)|^{3}} - \frac{1}{|x_{-}(\gamma)|^{3}} \Big) d\eta.
$$
We successively bound the $ \Gamma_{2j}$ as follows.
\begin{align*}
	| \Gamma_{21}| &\leq  \| F(x)\|_{L^\infty}^2\frac{1}{2}\int_{\mathbb{T}} \frac{(|\partial_{\gamma}x_{-}(\gamma+h)|+|\partial_{\gamma}x_{-}(\gamma)|)|\partial_{\gamma}x_{-}(\gamma+h)-\partial_{\gamma}x_{-}(\gamma)| |\partial_{\gamma}x_{-}(\gamma)|}{|x_{-}(\gamma)|^{2}} d\eta\\
	&\lesssim \| F(x)\|_{L^\infty}^4\int_{\mathbb{T}} \int_{0}^{1}\frac{\|\dg x\|_{C^{\frac{1}{2}+\mu}}^{2}|h|^{\frac{1}{2}+\mu}|\partial_{\gamma}^{2}x(\gamma-r\eta)|}{|\eta|^{\frac{1}{2}-\mu}} dr d\eta\\
	&\lesssim \| F(x)\|_{L^\infty}^4|h|^{\frac{1}{2}+\mu}\|\dg x\|_{C^{\frac{1}{2}+\mu}}^{2}\|x\|_{\dot{H}^{2}}.
\end{align*}
\begin{align*}
	| \Gamma_{22}| &\leq \| F(x)\|_{L^\infty}^2\int_{\mathbb{T}} \frac{|\partial_{\gamma}x_{-}(\gamma+h)|^{2} |x_{-}(\gamma+h)- x_{-}(\gamma)||\partial_{\gamma}x_{-}(\gamma)|}{|x_{-}(\gamma)|^{3}} d\eta\\
	&\lesssim  \| F(x)\|_{L^\infty}^5 \int_{\mathbb{T}} \int_{0}^{1} |h|\frac{\|\dg x\|_{C^{\frac{1}{2}+\mu}}^{3} |\dg x_{-}(\gamma+rh)|}{|\eta|^{\frac{3}{2}-3\mu}} dr d\eta\\
	&\lesssim  \| F(x)\|_{L^\infty}^5\int_{\mathbb{T}} \int_{0}^{1} |h|\frac{\|\dg x\|_{C^{\frac{1}{2}+\mu}}^{4}}{|\eta|^{1-4\mu}} dr d\eta\lesssim  \| F(x)\|_{L^\infty}^5|h|\|\dg x\|_{C^{\frac{1}{2}+\mu}}^{4}.
\end{align*}
\begin{align*}
	| \Gamma_{23}| &\le\| F(x)\|_{L^\infty}^2\int_{\mathbb{T}} \frac{|\partial_{\gamma}x_{-}(\gamma+h)|^{2} |x_{-}(\gamma+h)| |\partial_{\gamma}x_{-}(\gamma+h)-\partial_{\gamma}x_{-}(\gamma)|}{|x_{-}(\gamma)|^{3}} d\eta\\
	&\lesssim \| F(x)\|_{L^\infty}^5\|\dg x\|_{C^{\frac{1}{2}+\mu}}^{3}\|x\|_{C^{1}} |h|^{\frac{1}{2}+\mu} \int_{\mathbb{T}} \frac{d\eta}{|\eta|^{1-2\mu}} \\
	&\les\| F(x)\|_{L^\infty}^5\|\dg x\|_{C^{\frac{1}{2}+\mu}}^{3} |h|^{\frac{1}{2}+\mu}.
\end{align*}
Finally, for $a\in (0, 1)$ satisfying $(\mez+\mu)a<4\mu$, we have
\begin{align*}
	| \Gamma_{24}| &\leq\| F(x)\|_{L^\infty}^2 \frac{1}{2}\int_{\mathbb{T}} |\partial_{\gamma}x_{-}(\gamma+h)|^{2}  |x_{-}(\gamma+h)||\partial_{\gamma}x_{-}(\gamma+h)| |x_{-}(\gamma+h)-x_{-}(\gamma)| \\ &\hspace{0.5in}\times \frac{1}{|x_{-}(\gamma+h)||x_{-}(\gamma)|} \Big(\frac{1}{|x_{-}(\gamma+h)|^{2}} +\frac{1}{|x_{-}(\gamma+h)||x_{-}(\gamma)|} + \frac{1}{|x_{-}(\gamma)|^{2}} \Big) d\eta\\
	&\lesssim \| F(x)\|_{L^\infty}^5\int_{\mathbb{T}} \|\partial_{\gamma}x\|_{C^{\frac{1}{2}+\mu}}^{3}  \|x\|_{C^{1}} |x_{-}(\gamma+h)-x_{-}(\gamma)|^{a}|x_{-}(\gamma+h)-x_{-}(\gamma)|^{1-a}\frac{d\eta}{|\eta|^{\frac{3}{2} -3\mu}},
\end{align*}
and therefore
\begin{align*}
|\Gamma_{24}|	&\lesssim \| F(x)\|_{L^\infty}^5|h|^{(\frac{1}{2}+\mu)a}\|x\|_{C^{1}}\|\dg x\|_{C^{\frac{1}{2}+\mu}}^{4}\int_{\mathbb{T}} \frac{d\eta}{|\eta|^{\frac{3}{2} -3\mu - (\frac{1}{2}+\mu)(1-a)}} \\
	&\lesssim \| F(x)\|_{L^\infty}^5|h|^{(\frac{1}{2}+\mu)a}\|x\|_{C^{1}}\|\dg x\|_{C^{\frac{1}{2}+\mu}}^{4}.
\end{align*}
Gathering the above estimates, we conclude the proof of Proposition \ref{prop:ld:Holder}.  
\end{proof}
	\subsubsection{Sobolev estimate for $\p_\g \ld$}
Our goal is to establish the following bound for $\| \dg\lambda\|_{\dot H^{\frac12}}$.
\begin{prop}
For any $\mu\in (0, \mez)$, there exists a polynomial $\cP(\cdot, \cdot)$ such that 
\bq\label{est:ld:Sobolev}
\| \p_\g \ld\|_{\dot H^\mez}\le \cP(\|F(x)\|_{L^{\infty}}, \| \p_\g x\|_{C^{\mez+\mu}})\| \p_\g^2x\|_{L^2}.
\eq
\end{prop}
 \begin{proof}
 We use the  decomposition \eqref{dgammalambda}  for $\p_\g \ld$ together with the double-integral definition of $\| \cdot\|_{\dot H^\mez}$ to have
\begin{align*}
	\| \dg\lambda\|_{\dot H^{\frac12}}^{2} &= \int_{\mathbb{T}}\int_{\mathbb{T}}\frac{|\dg\lambda(\gamma)- \dg\lambda(\gamma-\xi)|^{2}}{|\xi|^{2}} d\xi d\gamma\leq \sum_{i=1}^{3}\wt{\Gamma}_{1i} + \sum_{i=1}^{4}\wt{\Gamma}_{2i}
\end{align*}
where
$$\wt{\Gamma}_{11} =\frac{1}{A(t)} \int_{\mathbb{T}}d\gamma\int_{\mathbb{T}}\frac{d\xi}{|\xi|^{2}} \Big|\int_{\mathbb{T}} \frac{\dg x_{-}(\gamma)}{|x_{-}(\gamma)|} \cdot (\dg^{2}x(\gamma-\eta) - \dg^{2}x(\gamma-\eta-\xi)) d\eta \Big|^{2},
$$
$$
\wt{\Gamma}_{12} = \frac{1}{A(t)}\int_{\mathbb{T}}d\gamma\int_{\mathbb{T}}\frac{d\xi}{|\xi|^{2}} \Big|\int_{\mathbb{T}} \frac{\dg x_{-}(\gamma)-\dg x_{-}(\gamma-
	\xi)}{|x_{-}(\gamma)|} \cdot  \dg^{2}x(\gamma-\eta-\xi) d\eta \Big|^{2},$$
$$\wt{\Gamma}_{13} =\frac{1}{A(t)} \int_{\mathbb{T}}d\gamma\int_{\mathbb{T}}\frac{d\xi}{|\xi|^{2}} \Big|\int_{\mathbb{T}} \dg x_{-}(\gamma-\xi) \cdot  \dg^{2}x(\gamma-\eta-\xi) \Big(\frac{1}{|x_{-}(\gamma)|} - \frac{1}{|x_{-}(\gamma-\xi)|}\Big) d\eta \Big|^{2},
$$
and
$$\wt{\Gamma}_{21} = \frac{1}{2A(t)} \int_{\mathbb{T}}d\gamma\int_{\mathbb{T}}\frac{d\xi}{|\xi|^{2}} \Big|\int_{\mathbb{T}} \frac{(|\partial_{\gamma}x_{-}(\gamma)|^{2}-|\partial_{\gamma}x_{-}(\gamma-\xi)|^{2}) x_{-}(\gamma)\cdot\partial_{\gamma}x_{-}(\gamma)}{|x_{-}(\gamma)|^{3}} d\eta \Big|^{2},
$$
$$\wt{\Gamma}_{22} = \frac{1}{2A(t)} \int_{\mathbb{T}}d\gamma\int_{\mathbb{T}}\frac{d\xi}{|\xi|^{2}} \Big|\int_{\mathbb{T}} \frac{|\partial_{\gamma}x_{-}(\gamma-\xi)|^{2} (x_{-}(\gamma)- x_{-}(\gamma-\xi))\cdot\partial_{\gamma}x_{-}(\gamma)}{|x_{-}(\gamma)|^{3}} d\eta \Big|^{2},$$	$$\wt{\Gamma}_{23} = \frac{1}{2A(t)} \int_{\mathbb{T}}d\gamma\int_{\mathbb{T}}\frac{d\xi}{|\xi|^{2}} \Big|\int_{\mathbb{T}}\frac{|\partial_{\gamma}x_{-}(\gamma-\xi)|^{2}  x_{-}(\gamma-\xi)\cdot(\partial_{\gamma}x_{-}(\gamma)-\partial_{\gamma}x_{-}(\gamma-\xi))}{|x_{-}(\gamma)|^{3}} d\eta \Big|^{2},$$
\begin{align*}
\wt{\Gamma}_{24} = \frac{1}{2A(t)} \int_{\mathbb{T}}d\gamma\int_{\mathbb{T}}\frac{d\xi}{|\xi|^{2}} \Big|\int_{\mathbb{T}}|\partial_{\gamma}x_{-}(\gamma-\xi)|^{2}&  x_{-}(\gamma-\xi)\cdot\partial_{\gamma}x_{-}(\gamma-\xi)\\
&\times\Big(\frac{1}{|x_{-}(\gamma)|^{3}} - \frac{1}{|x_{-}(\gamma-\xi)|^{3}} \Big) d\eta \Big|^{2}.
\end{align*}
Using
$$\dg^{2}x(\gamma-\eta) - \dg^{2}x(\gamma-\eta-\xi) = \partial_{\eta}\left\{\dg x(\gamma)- \dg x(\gamma-\eta)- \dg x(\gamma-\xi)+\dg x(\gamma-\eta-\xi)\right\}$$
we can integrate by parts to obtain
$$\wt{\Gamma}_{11} \leq \wt{\Gamma}_{111} + \wt{\Gamma}_{112}$$
where
\bq\label{wtT111}
\begin{aligned}\wt{\Gamma}_{111}& =\frac{1}{A(t)} \int_{\mathbb{T}}d\gamma\int_{\mathbb{T}}\frac{d\xi}{|\xi|^{2}} \Big|\int_{\mathbb{T}} \frac{\dg^{2} x(\gamma-\eta)}{|x_{-}(\gamma)|} \\
&\quad\quad\cdot \left\{\dg x(\gamma)- \dg x(\gamma-\eta)- \dg x(\gamma-\xi)+\dg x(\gamma-\eta-\xi)\right\} d\eta \Big|^{2}
\end{aligned}
\eq
and\begin{align*}
\wt{\Gamma}_{112} &= \frac{1}{A(t)}\int_{\mathbb{T}}d\gamma\int_{\mathbb{T}}\frac{d\xi}{|\xi|^{2}} \Big|\int_{\mathbb{T}} \frac{[x_{-}(\gamma)\cdot \dg x(\gamma-\eta)] \dg x_{-}(\gamma)}{|x_{-}(\gamma)|^{3}}\\
&\quad\quad \cdot \left\{\dg x(\gamma)- \dg x(\gamma-\eta)- \dg x(\gamma-\xi)+\dg x(\gamma-\eta-\xi)\right\} d\eta \Big|^{2}.
\end{align*}
Now, for any $b\in [0, 1]$ we have
$$|\dg x(\gamma)- \dg x(\gamma-\eta)- \dg x(\gamma-\xi)+\dg x(\gamma-\eta-\xi)| \leq |\xi|^{(\frac{1}{2}+\mu)b}|\eta|^{(\frac{1}{2}+\mu)(1-b)} \|\dg x\|_{C^{\frac{1}{2}+\mu}}.$$
Hence, for $\frac{1}{1+2\mu} < b < 1$ we obtain
\begin{align*}
	\wt{\Gamma}_{111} &\leq \|F(x)\|_{L^{\infty}}^2\int_{\mathbb{T}}d\gamma\int_{\mathbb{T}}\frac{d\xi}{|\xi|^{2-(1+2\mu)b}} \Big|\int_{\mathbb{T}} |\dg^{2} x(\gamma-\eta)|\|\dg x\|_{C^{\frac{1}{2}+\mu}}\|F(x)\|_{L^{\infty}} \frac{d\eta}{|\eta|^{1-(\frac{1}{2}+\mu)(1-b)}} \Big|^{2}\\
	&\les \|F(x)\|_{L^{\infty}}^4\|\dg x\|_{C^{\frac{1}{2}+\mu}}^{2}\Big\| \dg^{2}x \ast |\cdot|^{-1+(\frac{1}{2}+\mu)(1-b)}\Big\|_{L^{2}}^{2}\\
	&\les  \|F(x)\|_{L^{\infty}}^4\|\dg x\|_{C^{\frac{1}{2}+\mu}}^{2} \|\p_\g^2x\|_{L^2}^{2} .
\end{align*}
On the other hand, writing 
\[
\dg x(\gamma)- \dg x(\gamma-\eta)- \dg x(\gamma-\xi)+\dg x(\gamma-\eta-\xi)=\eta\int_0^1\dg x(\gamma-r\eta)dr- \eta\int_0^1\dg x(\gamma-\xi-r\eta)dr
\]
yields
\begin{align*}
	\wt{\Gamma}_{112} &\les \|F(x)\|_{L^{\infty}}^{4}\int_{\mathbb{T}}d\gamma\int_{\mathbb{T}}\frac{d\xi}{|\xi|^{2-(1+2\mu)b}} \Big(\int_{\mathbb{T}}d\eta\int_{0}^{1}dr \|\dg x\|_{L^{\infty}} \|\dg x\|_{C^{\frac{1}{2}+\mu}}
	 \frac{ |\dg^{2} x(\gamma-r\eta)|}{|\eta|^{1-(\frac{1}{2}+\mu)(1-b)}} \Big)^{2}\\
&\les\|\dg x\|^2_{L^{\infty}} \|\dg x\|^2_{C^{\frac{1}{2}+\mu}}\|F(x)\|_{L^{\infty}}^{6}\int_{\mathbb{T}}d\gamma \Big(\int_{\mathbb{T}}\int_{0}^{1} \frac{ |\dg^{2} x(\gamma-\xi)|}{|\xi|^{1-(\frac{1}{2}+\mu)(1-b)}} \frac{dr}{r^{(\frac{1}{2}+\mu)(1-b)}} d\xi \Big)^{2}\\
	&\les\|\dg x\|_{L^{\infty}}^2 \|\dg x\|_{C^{\frac{1}{2}+\mu}}^2\|F(x)\|_{L^{\infty}}^{6}\Big\||\dg^{2}x|\ast |\cdot|^{-1+(\frac{1}{2}+\mu)(1-b)})\Big\|_{L^{2}_\g}^2\\
	& \les\|\dg x\|_{L^{\infty}}^2 \|\dg x\|_{C^{\frac{1}{2}+\mu}}^2\|F(x)\|_{L^{\infty}}^{6}\|\p_\g^2x\|_{L^2}^{2},
\end{align*}
provided that $\frac{1}{1+2\mu} < b < 1$.

We note that $\wt{\Gamma}_{12}$ is similar to $\wt{\Gamma}_{111}$ given by \eqref{wtT111} since 
$$\wt{\Gamma}_{12} = \int_{\mathbb{T}}d\gamma\int_{\mathbb{T}}\frac{d\xi}{|\xi|^{2}} \Big|\int_{\mathbb{T}} \frac{\dg x(\gamma)- \dg x(\gamma-\eta)- \dg x(\gamma-\xi)+\dg x(\gamma-\eta-\xi)}{|x_{-}(\gamma)|} \cdot \dg^2 x(\gamma-\eta-\xi) d\eta \Big|^{2}.$$
Regarding $\wt{\Gamma}_{13}$, the bound
\begin{align*}
	\Big|\frac{1}{|x_{-}(\gamma)|} - \frac{1}{|x_{-}(\gamma-\xi)|}\Big| 
	&\leq \frac{|x_{-}(\gamma)-x_{-}(\gamma-\xi)|}{|x_{-}(\gamma)||x_{-}(\gamma-\xi)|}\leq \|\dg x\|_{L^{\infty}}\|F(x)\|_{L^{\infty}}^{2} \frac{|\xi|^{b}}{|\eta|^{1+b}}
\end{align*}
implies
\begin{align*}
	\wt{\Gamma}_{13} &\leq \int_{\mathbb{T}}d\gamma\int_{\mathbb{T}}\frac{d\xi}{|\xi|^{2-2b}} \Big|\int_{\mathbb{T}} \|\dg x\|_{C^{\mez+\mu}}  \|\dg x\|_{L^{\infty}}\|F\|_{L^{\infty}}^{2} \frac{| \dg^{2}x(\gamma-\eta-\xi) |}{|\eta|^{\mez+b-\mu}} d\eta \Big|^{2}\\
	&\les \|\dg x\|_{C^{\mez+\mu}}^2  \|\dg x\|_{L^{\infty}}^2\|F(x)\|_{L^{\infty}}^4 \|\p^2_\g x\|_{L^2}^{2},
\end{align*}	
provided that  $\mez<b < \mez+\mu$. $\wt{\Gamma}_{21}$ can be controlled as
\begin{align*}
|\wt{\Gamma}_{21}| &\lesssim\|F(x)\|_{L^{\infty}}^{2} \int_{\mathbb{T}}d\gamma\int_{\mathbb{T}}\frac{d\xi}{|\xi|^{2}} \Big|\int_{\mathbb{T}} (|\partial_{\gamma}x_{-}(\gamma)|+|\partial_{\gamma}x_{-}(\gamma-\xi)|)(|\partial_{\gamma}x_{-}(\gamma)-\partial_{\gamma}x_{-}(\gamma-\xi)|)\\
&\qquad\qquad\qquad\qquad\times \|F(x)\|_{L^{\infty}}^{2} \|\partial_{\gamma}x\|_{C^{\frac{1}{2}+\mu}} \frac{d\eta}{|\eta|^{\frac{3}{2}-\mu}} \Big|^{2}\\
&\lesssim \|F(x)\|_{L^{\infty}}^6\|\partial_{\gamma}x\|_{C^{\frac{1}{2}+\mu}}^{6}\int_{\mathbb{T}}d\gamma\int_{\mathbb{T}}\frac{d\xi}{|\xi|^{2}} \Big|\int_{\mathbb{T}} \frac{|\xi|^{\frac{1}{2}+\mu}}{|\eta|^{1-2\mu}} d\eta \Big|^{2}\\
&\les\|F(x)\|_{L^{\infty}}^6\|\partial_{\gamma}x\|_{C^{\frac{1}{2}+\mu}}^{6}.
\end{align*}
Next, for  $\frac{1}{2} < b < \frac{1}{2} + 3\mu$ we estimate
\begin{align*}
|\wt{\Gamma}_{22}| &\lesssim \|F(x)\|_{L^{\infty}}^2\int_{\mathbb{T}}d\gamma\int_{\mathbb{T}}\frac{d\xi}{|\xi|^{2}} \Big|\int_{\mathbb{T}}|x_{-}(\gamma)- x_{-}(\gamma-\xi)|^{b}|x_{-}(\gamma)- x_{-}(\gamma-\xi)|^{1-b}\\
&\qquad\qquad\qquad\qquad\times\|\partial_{\gamma}x\|_{C^{\frac{1}{2}+\mu}}^{3}\|F(x)\|_{L^{\infty}}^3\frac{d\eta}{ |\eta|^{\frac{3}{2}-3\mu}}\Big|^{2}\\
&\les \|F(x)\|_{L^{\infty}}^8\int_{\mathbb{T}}d\gamma\int_{\mathbb{T}}\frac{d\xi}{|\xi|^{2}} \Big|\int_{\mathbb{T}} \frac{|\xi|^{b}\|x\|_{C^{1}}\|\partial_{\gamma}x\|_{C^{\frac{1}{2}+\mu}}^{3}}{|\eta|^{\frac{1}{2}-3\mu + b }} d\eta \Big|^{2}\\
&\lesssim  \|F(x)\|_{L^{\infty}}^8\|\partial_{\gamma}x\|_{C^{\frac{1}{2}+\mu}}^{8}.
\end{align*}
$\wt{\Gamma}_{23}$ can be bounded as
\begin{align*}
|\wt{\Gamma}_{23}| &\lesssim  \|F(x)\|_{L^{\infty}}^2\int_{\mathbb{T}}d\gamma\int_{\mathbb{T}}\frac{d\xi}{|\xi|^{1-2\mu}} \Big|\int_{\mathbb{T}}\|\partial_{\gamma}x\|_{C^{\frac{1}{2}+\mu}}^{3}\|x\|_{C^{1}}\|F(x)\|_{L^{\infty}}^3\frac{d\eta}{|\eta|^{1-2\mu}} \Big|^{2}\\
&\lesssim \|F(x)\|_{L^{\infty}}^8\|\partial_{\gamma}x\|_{C^{\frac{1}{2}+\mu}}^{8}.
\end{align*}
Finally,  for $\frac{1}{2} < b < \frac{1}{2} + 3\mu$ we have
\begin{align*}
|\wt{\Gamma}_{24}| &\lesssim \|F(x)\|_{L^{\infty}}^2 \int_{\mathbb{T}}d\gamma\int_{\mathbb{T}}\frac{d\xi}{|\xi|^{2}} \Big(\int_{\mathbb{T}}|\partial_{\gamma}x_{-}(\gamma-\xi)|^{2}  |x_{-}(\gamma-\xi)||\partial_{\gamma}x_{-}(\gamma-\xi)||x_{-}(\gamma)-x_{-}(\gamma-\xi)|
\\
&\hspace{0.5in}\times \frac{1}{|x_{-}(\gamma)||x_{-}(\gamma-\xi)|} \Big(\frac{1}{|x_{-}(\gamma-\xi)|^{2}} +\frac{1}{|x_{-}(\gamma-\xi)||x_{-}(\gamma)|} + \frac{1}{|x_{-}(\gamma)|^{2}} \Big)d\eta \Big)^{2}\\
&\lesssim \|F(x)\|_{L^{\infty}}^2\int_{\mathbb{T}}d\gamma\int_{\mathbb{T}}\frac{d\xi}{|\xi|^{2}} \Big(\int_{\mathbb{T}}\|F(x)\|_{L^{\infty}}^3\|\partial_{\gamma}x\|_{C^{\frac{1}{2}+\mu}}^{3} \\
&\qquad\qquad\qquad\qquad\quad\times |x_{-}(\gamma)-x_{-}(\gamma-\xi)|^{a}|x_{-}(\gamma)-x_{-}(\gamma-\xi)|^{1-a} \frac{d\eta}{|\eta|^{\frac{3}{2}-3\mu}} \Big)^{2},
\end{align*}
and therefore
\begin{align*}
|\wt{\Gamma}_{24}|&\lesssim \|F(x)\|_{L^{\infty}}^8\int_{\mathbb{T}}d\gamma\int_{\mathbb{T}}\frac{d\xi}{|\xi|^{2-2a}} \Big(\int_{\mathbb{T}}\|\partial_{\gamma}x\|_{C^{\frac{1}{2}+\mu}}^{3}  \|x\|_{C^{1}} \frac{d\eta}{|\eta|^{\frac{1}{2}-3\mu + a}} \Big)^{2}\\
&\lesssim \|F(x)\|_{L^{\infty}}^8\|\partial_{\gamma}x\|_{C^{\frac{1}{2}+\mu}}^8.
\end{align*}
Gathering the above estimates lead to the desired  bound \eqref{est:ld:Sobolev}.
\end{proof}
\subsection{Estimates for tangential terms $T_j$}

Recall that 
\begin{align}
&T_{1} = \int_{\mathbb{T}} \Lambda^{2s}\dg^{2}x(\gamma) \lambda(\gamma) \cdot\dg^{3} x(\gamma)d\gamma,\\
&T_{2} = 2\int_{\mathbb{T}} \Lambda^{2s}\dg^{2}x(\gamma) \dg \lambda(\gamma)\cdot \dg^{2} x(\gamma) d\gamma,\\
&T_{3} = \int_{\mathbb{T}} \Lambda^{2s}\dg^{2}x(\gamma) \dg^{2}\lambda(\gamma) \cdot\dg x(\gamma) d\gamma.
\end{align}
We first rewrite $T_1$ using commutator and integration by parts as
\begin{align*}
	T_{1} 
	&= \int_{\mathbb{T}} \Lambda^{s}\dg^{2}x(\gamma)  \lambda(\gamma)\cdot \Lambda^{s}\dg^{3} x(\gamma)d\gamma + \int_{\mathbb{T}} \Lambda^{s}\dg^{2}x(\gamma) \cdot[\Lambda^{s}, \lambda(\gamma)] \dg^{3} x(\gamma)d\gamma \\
	 &=-\mez\int_{\mathbb{T}} \p_\g\ld   |\Lambda^{s}\dg^{2} x(\gamma)|^2d\gamma + \int_{\mathbb{T}} \Lambda^{s}\dg^{2}x(\gamma)\cdot [\Lambda^{s}, \lambda(\gamma)] \dg^{3} x(\gamma)d\gamma \\
	& := T_{11} +T_{12}.
\end{align*}
We fix $\mu\in (0, \mez)$ and denote by $\cP(\cdot, \cdot)$  polynomials that depend only on $\mu$ . It follows from the $L^\infty$ bound \eqref{dld:Linfty} for $\p_\g \ld$ that
\bq\label{est:T11}
|T_{11}|\le \cP(\| F(x)\|_{L^\infty}, \| \p_\g x\|_{C^{\mez+\mu}})\| \p_\g^2 x\|_{L^2}\| \p_\g^2 x\|_{H^s}^2.
\eq
The commutator estimate \eqref{cmt:1}  gives
\[
|T_{12}| \leq \|\Lambda^{s}\dg^{2}x\|_{L^{2}}\|[\Lambda^{s}, \lambda(\gamma)] \dg^{3} x(\gamma))\|_{L^{2}} \les \|\Lambda^{s}\dg^{2}x\|_{L^{2}}^{2}\|\dg\lambda\|_{C^{\eps}}
\]
for any $\eps>0$.  A combination \eqref{dld:Linfty} and \eqref{lambdaCepsilon} implies that for any $\eps\in (0, \mu)$, 
\[
\|\dg\lambda\|_{C^{\eps}}\le \cP(\| F(x)\|_{L^\infty}, \| \p_\g x\|_{C^{\mez+\mu}})\| \p_\g^2x\|_{H^\mu},
\]
whence
\bq\label{est:T12}
|T_{12}|\le \cP(\| F(x)\|_{L^\infty}, \| \p_\g x\|_{C^{\mez+\mu}})\| \p_\g^2x\|_{H^\mu}\|\Lambda^{s}\dg^{2}x\|_{L^{2}}^{2}.
\eq
As for $T_2$, we note
$$T_{2} = \int_{\mathbb{T}} \Lambda^{s}\dg^{2}x(\gamma) \cdot\Lambda^{s}(\dg \lambda(\gamma) \dg^{2} x(\gamma))) d\gamma \leq \|\Lambda^{s}\dg^{2}x\|_{L^{2}}\|\Lambda^{s}(\dg \lambda \dg^{2}x)\|_{L^{2}}.$$
Applying \eqref{BonyR}, \eqref{pp:1} and \eqref{pp:2} gives the product estimate
\bq\label{productest:2}
\| uv\|_{H^s}\les \| u\|_{L^\infty}\| v\|_{H^s}+\| u\|_{H^\mez}\| v\|_{B^{s-\mez}_{\infty, \infty}}\quad\forall s \in (0, \mez).
\eq
Using \eqref{productest:2} with $u=\p_\g \ld$, $v=\p_\g^2 x$ and recalling the embedding $H^s\subset B^{s-\mez}_{\infty, \infty}$, we deduce 
\[
|T_2|\les \|\Lambda^{s}\dg^{2}x\|_{L^2}^2\big(\| \p_\g\ld\|_{L^\infty}+\| \p_\g\ld\|_{H^\mez} \big).
\]
Then in view of \eqref{dld:Linfty} and \eqref{est:ld:Sobolev}, we obtain 
\bq\label{est:T2}
|T_2| \le\cP(\| F(x)\|_{L^\infty}, \| \p_\g x\|_{H^{1+\mu}})\|\Lambda^{s}\dg^{2}x\|_{L^{2}}^{2}.
 \eq
Now, since $0=\p_\gamma|\p_\gamma x|^2=2\p_\gamma x\p^2_\gamma x$ we have
\begin{equation}\label{T3:cmtform}
	T_{3} = -\int_{\mathbb{T}}\dg^{2}\lambda(\gamma)\cdot[\Lambda^{2s},  \dg x(\gamma) d\gamma]\dg^{2}x(\gamma) d\gamma.
\end{equation}
In order to control the commutator, we prove
\begin{lemm}
For $s\in (0, \mez)$ we have
\bq\label{cmt:2}
	\| [\Lambda^{2s}, \dg x]\dg^2x\|_{H^\frac{1}{2}}\le C\| \dg x\|_{C^{\mez+s}}\| \dg^{2} x\|_{H^s}.
\eq
\end{lemm}
\begin{proof}
Note that $u=\dg x\in H^{1+s}\subset C^{\frac{1}{2}+s}$ and $v=\dg^{2}x\in H^s$. Using the Bony decomposition $uv=T_uv+T_vu+R(u, v)$ we write
\[
[\Lambda^{2s}, u]v=[\Lambda^{2s}, T_u]v+\Lambda^{2s}(T_vu)+\Lambda^{2s}R(u, v)-T_{\Lambda^{2s}h}u-R(u, \Lambda^{2s}v).
\]
By virtue of Lemma \ref{comt:LdT}, 
\[
\| [\Lambda^{2s}, T_u]v\|_{H^\frac{1}{2}}\lesssim \| u\|_{B^{\frac12+s}_{\infty, \infty}}\| v\|_{H^s}.
\]
On the other hand, paraproduct rule \eqref{pp:2}  yields
\begin{align*}
	&\| \Lambda^{2s}(T_vu)\|_{H^\frac{1}{2}}\lesssim \|T_vu\|_{H^{2s+\frac 12}}\lesssim \| v\|_{B^{s-\frac 12}_{\infty, \infty}}\| u\|_{H^{1+s}},\\
	&\|T_{\Lambda^{2s}v}u\|_{H^\frac{1}{2}}\lesssim \|\Lambda^{2s}v\|_{B^{-s-\frac12}_{\infty, \infty}}\|u\|_{H^{1+s}}\les \| v\|_{B^{s-\frac 12}_{\infty, \infty}}\| u\|_{H^{1+s}}.
\end{align*}
The Bony remainder $R(\cdot, \cdot)$ can be estimated using \eqref{BonyR} as
\begin{align*}
	&\| \Lambda^{2s}R(u, v)\|_{H^\frac12}\lesssim \| R(u, v)\|_{H^{2s+\frac12}}\lesssim  \| v\|_{B^{s-\frac 12}_{\infty, \infty}}\| u\|_{H^{1+s}},\\
	&\| R(u, \Lambda^{2s}v)\|_{H^\frac12}\lesssim  \|\Lambda^{2s}v\|_{B^{-s-\frac 12}_{\infty, \infty}}\| u\|_{H^{1+s}}\les \|v\|_{B^{s-\frac 12}_{\infty, \infty}}\| u\|_{H^{1+s}}.
\end{align*}
Finally, Lemma \ref{equi:HZ} implies 
\begin{align*}
&\| u\|_{B^{\frac12+s}_{\infty, \infty}}\simeq \| \p_\g x\|_{C^{\frac12+s}},\quad \|v\|_{B^{s-\frac 12}_{\infty, \infty}}\le \|\p_\g x\|_{B^{s+\frac 12}_{\infty, \infty}}\simeq \|\p_\g x\|_{C^{s+\frac 12}},
\end{align*}
 so that the above estimates lead to \eqref{cmt:2}.
\end{proof}
Using  \eqref{cmt:2} and \eqref{est:ld:Sobolev}, we deduce from \eqref{T3:cmtform} that 
\bq\label{est:T3}
\begin{aligned}
|T_3|&\lesssim \| \dg^2\lambda\|_{H^{-\frac12}} \| \dg x\|_{C^{\mez+s}}\| \dg^{2} x\|_{H^s}\\
& \le \cP(\|F(x)\|_{L^{\infty}}, \| \p_\g x\|_{C^{\mez+\mu}})\| \p_\g^2x\|_{L^2} \| \dg x\|_{C^{\mez+s}}\| \dg^{2} x\|_{H^s}.
\end{aligned}
\eq
From \eqref{est:T11}, \eqref{est:T12}, \eqref{est:T2}, \eqref{est:T3}, we deduce
\begin{prop}\label{prop:Tj}
	For all $s\in (0, \mez)$, there exists a polynomial $\cP(\cdot, \cdot)$ such that 
	\bq\label{est:Tj}
	\sum_{j=1}^3|T_j|\le \cP(\| F(x)\|_{L^\infty}, \| \p_\g^2x\|_{H^s})\| \p_\g^2x\|_{H^s}^2.
	\eq
	\end{prop}
	Next, for the $L^2$ bound, we multiply \eqref{SQGp} by $x(\g)$, integrate, symmetrize then integrate by parts: 
	\[
	\begin{aligned}
	\mez\frac{d}{dt}\| x\|_{L^2}^2&=\mez \int_\T\int_{\T} \frac{\p_\g x_-(\g, \eta)\cdot x_-(\g, \eta)}{|x_-(\g, \eta)|}d\eta d\g+\int_\T\ld(\g)\p_\g x(\g)\cdot x(\g)d\g\\
	&=-\frac14\int_{\T} \int_\T |x_-(\g, \eta)|^2\p_\g\frac{1}{|x_-(\g, \eta)|}d\eta d\g-\mez \int_\T\p_\g\ld(\g) |x(\g)|^2d\g.
	\end{aligned}
	\]
	Clearly, 
	\[
	\int_\T\int_{\T} |x_-(\g, \eta)|^2\p_\g\frac{1}{|x_-(\g, \eta)|}d\eta d\g =-\int_\T \int_{\T} \p_\g|x_-(\g, \eta)|d\g d\eta =0.
	\]
An application of \eqref{dld:Linfty} gives
\bq\label{est:L2}
\frac{d}{dt}\| x\|_{L^2}^2\le  \cP(\| F(x)\|_{L^\infty}, \| \p_\g^2x\|_{H^s})\| x\|_{L^2}^2.
\eq
Combining Propositions \ref{prop:Nj}, \ref{prop:Tj} and \eqref{est:L2}, we conclude the a priori estimate for the $H^{2+s}$ norm.
\begin{prop}\label{prop:apriorinorm}
For all $s\in (0, \mez)$, there exists a polynomial $\cP(\cdot, \cdot)$ such that 
	\bq\label{est:norm}
	\frac{d}{dt}\| x(t)\|^2_{H^{2+s}}\le \cP(\| F(x)(t)\|_{L^\infty}, \| \p_\g^2x(t)\|_{H^s})\| x(t)\|_{H^{2+s}}^2,\quad t\in [0, T].
	\eq
\end{prop}
\section{Propagation of the arc chord condition}\label{sec:arcchord}
	We assume throughout this section  that $x\in C([0, T]; H^3)$ is a solution of \eqref{SQGp}-\eqref{ld}. In order to close the a priori estimates \eqref{est:Nj}, \eqref{est:Tj} and \eqref{est:L2}, it remains to bound the arc chord condition in terms of $\| x\|_{H^{2+s}}$ and itself.  To that end, we differentiate  $F(x)(\g, \eta, t)=\frac{|\eta|}{|x_-(\g, \eta, t)|}$ 
to have
	\bq
	\begin{aligned}
		\p_t F(\g, \eta)&=-|\eta|\frac{x_-(\g, \eta)\cdot \p_t x_-(\g, \eta)}{|x_-(\g, \eta)|^3}\\
		&=-|\eta|\frac{x_-(\g, \eta)\cdot [Q(\g)-Q(\g-\eta)]}{|x_-(\g, \eta)|^3}-|\eta|\frac{x_-(\g, \eta)\cdot \p_\g x(\g) [\ld(\g)-\ld (\g-\eta)]}{|x_-(\g, \eta)|^3}\\
		&\quad-|\eta|\frac{x_-(\g, \eta)\cdot \p_\g x_-(\g, \eta)\ld (\g-\eta)}{|x_-(\g, \eta)|^3}\\
		&=-F_1-F_2-F_3,
	\end{aligned}
	\eq
	where 
	\[
	Q(\g)=\int_{\T} \p_\g x_-(\g, \xi)g(\g, \xi) d\xi,\quad g(\g, \xi)=\frac{1}{|x_-(\gamma, \xi)|}.
	\]
	It is readily seen that 
	\[
	|F_2|\le|\eta|\frac{\|\p_\g x\|_{L^\infty} |\eta|\| \p_\g \ld\|_{L^\infty}}{|x_-(\g, \eta)|^2}\le\|F(x)\|_{L^\infty}^2\|\p_\g x\|_{L^\infty}\| \p_\g \ld\|_{L^\infty},
	\]
	so that invoking \eqref{dld:Linfty} yields
	\bq\label{est:F2}
	|F_2|\le \cP(\|F(x)\|_{L^\infty}, \|\p_\g x\|_{C^{\mez+\mu}})\| \p^2_\g x\|_{L^2}\quad\forall \mu \in (0, \mez).
	\eq
	Recalling the identity \eqref{identity1} for $x_-\cdot \p_\g x_-$ we obtain 
	\bq\label{est:F3:0}
	|F_3|\les |\eta|\| F(x)\|_{L^\infty}^3\frac{|\eta|^2\| \p_\g x\|_{C^\mez}^2\| \ld\|_{L^\infty}}{|\eta|^3}\les \| F(x)\|_{L^\infty}^3\| \p_\g x\|_{C^\mez}^2\| \ld\|_{L^\infty}.
	\eq
	Then, appealing to \eqref{dld:Linfty} again, \eqref{est:F3:0} implies 
	\bq\label{est:F3}
	|F_3|\le\cP(\|F(x)\|_{L^\infty}, \|\p_\g x\|_{C^{\mez+\mu}})\| \p^2_\g x\|_{L^2}\quad\forall \mu \in (0, \mez).
	\eq
	For the most difficult term $F_1$ we estimate
	\[
	|F_1|\le \| F(x)\|_{L^\infty}^4\frac{1}{|\eta|^2}\big(\int_{|\xi|<|\eta|^2}\!+\!\int_{|\xi|>|\eta|^2}\big)\frac{1}{|\xi|}\big|x_-(\g, \eta)\cdot[ \p_\g x_-(\g, \xi)- \p_\g x_-(\g\!-\!\eta, \xi)]\big|d\xi:=F_{11}+F_{12}.
	\]
	By the mean value theorem and Cauchy-Schwarz's inequality, we have
	\[
	\begin{aligned}
		|F_{11}|&\le \| F(x)\|_{L^\infty}^4\frac{1}{|\eta|^2}\int_{|\xi|<|\eta|^2}\frac{1}{|\xi|}|\eta|\| \p_\g x\|_{L^\infty}|\xi|\int_0^1 \big(|\p_\g^2 x(\g- r\xi)|+|\p_\g^2 x(\g-\eta-r\xi)|\big)drd\xi\\
		&\le C\| F(x)\|_{L^\infty}^4\| \p_\g x\|_{L^\infty}\frac{1}{|\eta|}\int_0^1\frac{1}{\sqrt{r}}\Big( \int_{|\xi|<|\eta|^2}1 d\xi\Big)^\mez\Big(\int_{\T} |\p_\g^2 x(\xi)|^2d\xi\Big)^\mez dr\\
		&\le C\| F(x)\|_{L^\infty}^4\| \p_\g x\|_{L^\infty}\|\p_\g^2 x\|_{L^2}.
	\end{aligned}
	\]
	As for $F_{12}$ we split 
	\begin{align*}
		&x_-(\g, \eta)\cdot[ \p_\g x_-(\g, \xi)- \p_\g x_-(\g-\eta, \xi)]\\
		&=[x_-(\g, \eta)-\eta\p_\g x(\g)]\cdot\p_\g x_-(\g, \eta)+\eta\p_\g x(\g)\cdot \p_\g x_-(\g, \eta)\\
		&\quad-[x_-(\g, \eta)-\eta\p_\g x(\g)]\cdot\p_\g x_-(\g- \xi, \eta) -\eta\p_\g x(\g)\cdot\p_\g x_-(\g- \xi, \eta)
	\end{align*}
	and accordingly $F_{12}=\sum_{j=1}^4F_{12 j}$. Clearly,
	\begin{align*}
		|F_{121}|&\le \| F(x)\|_{L^\infty}^4\frac{1}{|\eta|^2}\int_{|\xi|>\eta^2}\frac{1}{|\xi|}|\eta|^{\tdm+s}\| \p_\g x\|_{C^{\mez+s}}|\eta|^{\mez+s}\| \p_\g x\|_{C^{\mez+s}}d\xi\\
		&\le 2\| F(x)\|_{L^\infty}^4\| \p_\g x\|_{C^{\mez+s}}^2|\eta|^{2s}|\ln \pi-\ln |\eta|^2|\\
		&\le C\| F(x)\|_{L^\infty}^4\| \p_\g x\|_{C^{\mez+s}}^2.
	\end{align*}
	In view of identity \eqref{dxdx-} and by the same argument, $F_{122}$ and $F_{123}$ obey the same bound as $F_{121}$. We further split $F_{124}=F_{1241}+F_{1242}$ according to  
	\begin{align*}
		\eta\p_\g x(\g)\cdot\p_\g x_-(\g- \xi, \eta)
		=\eta[\p_\g x(\g)-\p_\g x(\g-\xi)]\cdot \p_\g x_-(\g- \xi, \eta)+\eta \mez |\p_\g x_-(\g- \xi, \eta)|^2,
	\end{align*}
	where we have again appealed to \eqref{dxdx-}. By the mean value theorem and Cauchy-Schwarz's inequality,
	\begin{align*}
		|F_{1241}|&\le C\| F(x)\|_{L^\infty}^4\int_{\T}\frac{\| \p_\g x\|_{C^{\mez+s}}}{|\xi|^{\mez-s}}\int_0^1|\p_\g^2 x(\g-\xi-r\eta)|drd\xi\le C\| F(x)\|_{L^\infty}^4\| \p_\g x\|_{C^{\mez+s}}\| \p_\g^2 x\|_{L^2}.
	\end{align*}
	Finally, we bound 
	\begin{align*}
		|F_{1242}|&\le C\| F(x)\|_{L^\infty}^4\frac{1}{|\eta|}\int_{|\xi|>|\eta|^2}\frac{1}{|\xi|}|\eta|^{1+2s}\|\p_\g x\|_{C^{\mez+s}}^2d\xi\\
		&\le C\| F(x)\|_{L^\infty}^4\|\p_\g x\|_{C^{\mez+s}}^2|\eta|^{2s}|\ln \pi-\ln |\eta|^2|\\
		&\le C\| F(x)\|_{L^\infty}^4\|\p_\g x\|_{C^{\mez+s}}^2.
	\end{align*}
	We have proved 
	\bq
	|F_1|\le C\| F(x)\|_{L^\infty}^4\|\p_\g x\|_{H^{1+s}}^2.
	\eq
	and thus, in view of \eqref{est:F2} and \eqref{est:F3}, we obtain 
	\begin{prop}\label{prop:estF}
	For all $s\in (0, \mez)$, there exists $\cP(\cdot, \cdot)$ such that
	\bq\label{est:F}
	\left|\frac{d}{dt}F(x)(\g, \eta, t)\right|\le \cP(\|F(x)(t)\|_{L^\infty}, \|\p_\g x(t)\|_{C^{\mez+s}})\| \p^2_\g x(t)\|_{H^s},\quad t\in [0, T].
	\eq
	\end{prop}

	\section{Uniqueness}\label{sec:uniqueness}
	This section is devoted to the proof of the following stability result which implies the uniqueness of $H^{2+s}$ solutions. 
	\begin{theo}\label{theo:stability}
	Suppose that  $x$ and $y$ are two solutions of \eqref{SQGp}-\eqref{ld} in $L^\infty([0,T]; H^{2+s}(\T))$ and satisfy the arc chord condition. Then we have 
	\bq\label{stability}
	\|x(t)-y(t)\|^2_{H^1}+\big||\dg x(t)|-|\dg y(t)|\big|^2\le e^{Ct}\left(\|x(0)-y(0)\|^2_{H^1}+\big||\dg x(0)|-|\dg y(0)|\big|^2\right),
	\eq
	for any $t\in [0, T]$, where $C$ depends only on 
$\| (x, y)\|_{L^\infty([0, T]; H^{2+s}(\T))}$ and  $\| (F(x), F(y))\|_{L^\infty(\T\times \T\times [0, T])}$. Consequently, $x\equiv y$ if $x(0)=y(0)$. 
	\end{theo}
	As before we shall write $f=f(\g,t)$, $f'=f(\g-\e,t)$, $f_-=f-f'$ and $\int=\int_{\T}$ when there is no danger of confusion in writing  double integrals with respect to the variables $\g$ and $\eta$. Recall from \eqref{dld:Linfty}  that 
	\bq\label{ld:C1}
	\| (\ld(x), \ld(y))\|_{L^\infty([0, T]; W^{1, \infty}(\T))}\le C.
	\eq
	Set $z=x-y$.
	\subsection{$L^2$ estimates}
	We have
	$$
	\frac{1}{2}\frac{d}{dt}\|z\|^2_{L^2}=\int z\cdot \p_t z d\g= I_1+I_2,
	$$
	where
	$$
	I_1=\int z\cdot \int \Big(\frac{\dg x_-}{|x_-|}-\frac{\dg y_-}{|y_-|} \Big)d\e d\g,\quad
	I_2=\int z\cdot (\lambda(x)\dg x-\lambda(y)\dg y) d\g.
	$$
	We split $I_1=I_{11}+I_{12}$, where
	$$
	I_{11}=\int z\cdot \int \frac{\dg z_-}{|x_-|}d\e d\g,\quad
	I_{12}=\int z\cdot \int \dg y_-\Big(\frac{1}{|x_-|}-\frac{1}{|y_-|} \Big)d\e d\g.
	$$
	After symmetrizing, an integration by parts gives 
	$$
	I_{11}=\frac12\int\int\frac{z_-\cdot \dg z_-}{|x_-|} d\g d\e=\frac{1}{4}\int\int|z_-|^2\frac{x_-\cdot\dg x_-}{|x_-|^3} d\g d\e.
	$$
	The identity
	\begin{equation}\label{tvm} 
		f_-=\e \int_0^1\dg f(\g-r\e)dr
	\end{equation}
	together with H\"older's inequality (in $\g$) allows us to get the bound
	\begin{equation*}\label{I11}
		I_{11}\leq \|F(x)\|^2_{L^\infty}\|\dg x\|_{C^{\frac12}}\int_0^1dr\int\frac{d\e}{|\e|^{\mez}}\int d\g (|z|+|z'|)|\dg z(\g-r\e)|  \leq C\|z\|^2_{H^1}.
	\end{equation*}
	For $I_{12}$ one writes
	$$
	I_{12}=\int z\cdot \int \dg y_- \frac{|y_-|-|x_-|}{|x_-||y_-|}d\e d\g\leq \int\int \frac{|z||\dg y_-||z_-|}{|x_-||y_-|}d\e d\g,
	$$
	so that arguing as for $I_{12}$, we obtain
	$$
	I_{12}\leq \|F(x)\|_{L^\infty}\|F(y)\|_{L^\infty}\|\dg y\|_{C^{\frac12}}\int_0^1dr\int\frac{d\e}{|\e|^{\mez}}\int d\g |z| |\dg z(\g-r\e)|\leq C \|z\|^2_{H^1}.
	$$
	Thus,
	\begin{equation}\label{I1}
		I_{1}\leq C\|z\|^2_{H^1}.
	\end{equation}
	Next,   $I_2=I_{21}+I_{22}$, where
	$$
	I_{21}=\int z\cdot \dg z \lambda(x) d\g,\quad I_{22}=\int z\cdot\dg y  (\lambda(x)-\lambda(y)) d\g.
	$$
	It follows at once from \eqref{ld:C1} that $I_{21}
	\leq C \|z\|^2_{H^1}$. Similarly, $ I_{21}
	\leq C \|z\|^2_{H^1}$, provided that 
	\begin{equation}\label{boundsLambdaLinfty}
		\|\lambda(x)-\lambda(y)\|_{L^\infty}\leq C\|z\|_{H^1}.
	\end{equation}
	Taking this for granted, we obtain $I_{2}\leq C\|z\|^2_{H^1}$ which in conjunction with \eqref{I1} yields  the differential inequality
	\begin{equation}\label{L2evolution}
		\frac{d}{dt}\|z\|^2_{L^2}\leq C\|z\|^2_{H^1}.
	\end{equation}
	
	The remainder of this subsection is devoted to the proof  of \eqref{boundsLambdaLinfty}. Let us write
	\begin{equation}\label{SLambda}
		\lambda(x)-\lambda(y)=G_1+G_2,
	\end{equation} where
	$$
	G_1=\frac{\g+\pi}{2\pi}\int\Big[\frac{\dg x}{|\dg x|^2}\cdot \dg \Big(\int
	\frac{\dg x_-}{|x_-|}d\e \Big)-\frac{\dg y}{|\dg y|^2}\cdot \dg \Big(\int
	\frac{\dg y_-}{|y_-|}d\e \Big)\Big] d\g,
	$$
	and
	\begin{align*}
		G_2=&-\int_{-\pi}^\g \frac{\dg x(\e,t)}{|\dg x(\e,t)|^2}\cdot \de \Big(\int\frac{\dg
			x(\e,t)-\dg x(\e-\xi,t)}{|x(\e,t)-x(\e-\xi,t)|}d\xi \Big)d\e\\
		&+\int_{-\pi}^\g \frac{\dg y(\e,t)}{|\dg y(\e,t)|^2}\cdot \de \Big(\int\frac{\dg
			y(\e,t)-\dg y(\e-\xi,t)}{|y(\e,t)-y(\e-\xi,t)|}d\xi \Big)d\e.
	\end{align*}
	We shall prove that both $|G_1|$ and $|G_2|$ are bounded by $C\| z\|_{H^1}$. To that end,  we decompose further $G_1=G_{11}+G_{12}+G_{13}$, where
	\begin{align*}
	&G_{11}=\frac{\g+\pi}{2\pi}\int \frac{\dg z}{|\dg x|^2}\cdot \dg \Big(\int
	\frac{\dg x_-}{|x_-|}d\e \Big)d\g,\\
	&G_{12}=\frac{\g+\pi}{2\pi}\int \big(\frac{1}{|\dg x|^2}-\frac{1}{|\dg y|^2}\big)\dg y\cdot \dg \Big(\int
	\frac{\dg x_-}{|x_-|}d\e \Big)d\g,\\
	&	G_{13}=\frac{\g+\pi}{2\pi}\int \frac{\dg y}{|\dg y|^2}\cdot \dg \int\Big(
	\frac{\dg x_-}{|x_-|}-\frac{\dg y_-}{|y_-|}\Big)d\e d\g.
	\end{align*}
	A further splitting provides $G_{11}=G_{111}+G_{112}$ where
	\begin{align*}
	&G_{111}=-\frac{\g+\pi}{2\pi}\int \frac{\dg z}{|\dg x|^2}\cdot \int \dg x_-
	\frac{x_-\cdot\dg x_-}{|x_-|^3}d\e d\g,\\
	&G_{112}=\frac{\g+\pi}{2\pi}\int \frac{\dg z}{|\dg x|^2}\cdot \int
	\frac{\dg^2 x_-}{|x_-|}d\e d\g.
	\end{align*}
	The less singular term can be estimated analogously to $I_{11}$:
	$$|G_{111}|\leq \|F(x)\|^4_{L^\infty}\|\dg x\|_{C^{\frac12}}
	\int_0^1dr\int\frac{d\e}{|\e|^{\mez}}\int d\g |\dg z| |\dg^2 x(\g-r\e)|\leq C\|z\|_{H^1}.$$
	In $G_{112}$ we use  Cauchy-Schwarz's inequality as follows:
	\bq\label{G112}
	\begin{aligned}
	G_{112}&\leq\|F(x)\|^3_{L^\infty}\int\int {d\e}d\g  \frac{|\dg z(\g)|}{|\eta|^{\mez-s}}\frac{|\dg^2 x_-|}{|\eta|^{\mez+s}}\\
	&\les \|F(x)\|^3_{L^\infty}\| \p_\g^2x\|_{\dot H^s}\| \frac{|\dg z(\g)|}{|\eta|^{\mez-s}}\|_{L^2_{\g, \e}}\le C\| \p_\g z\|_{L^2}.
	\end{aligned}
	\eq
	Noticing 
	$$
	|G_{12}|\leq 2\|F(x)\|^2_{L^\infty}\int |\dg z|\Big|\dg \Big(\int
	\frac{\dg x_-}{|x_-|}d\e\Big)\Big|d\g,
	$$
	we  see that $G_{12}$ can be controlled analogously to $G_{11}$. 
	Regarding  $G_{13}$ we first integrate by parts using $\p_\g y\cdot \p_\g^2 y=0$ to obtain
	$$ 
	G_{13}=-\frac{\g+\pi}{2\pi}\int \frac{\dg^2 y}{|\dg y|^2}\cdot \int\Big(
	\frac{\dg x_-}{|x_-|}-\frac{\dg y_-}{|y_-|}\Big)d\e d\g=G_{131}+G_{132},
	$$ 
	where 
	\begin{align*}
	&G_{131}=-\frac{\g+\pi}{2\pi}\int \frac{\dg^2 y}{|\dg y|^2}\cdot \int
	\dg x_-\Big(\frac{1}{|x_-|}-\frac{1}{|y_-|}\Big)d\e d\g,\\
	&G_{132}=-\frac{\g+\pi}{2\pi}\int \frac{\dg^2 y}{|\dg y|^2}\cdot \int
	\frac{\dg z_-}{|y_-|} d\e d\g.
	\end{align*}
	$G_{131}$ can be controlled as in $I_{11}$:
	$$
	|G_{131}|\leq  \|F(y)\|^3_{L^\infty}\|F(x)\|_{L^\infty}\|\dg x\|_{C^{\frac12}}\int_0^1dr\int \frac{d\e}{|\e|^{\mez}}\int d\g |\dg^2 y||\dg z(\g-r\e)|\leq C\|z\|_{H^1}.
	$$
	After symmetrizing, 
	$$
	G_{132}=-\frac{\g+\pi}{4\pi}\int\int \frac{\dg^2 y_-}{|\dg y|^2}\cdot 
	\frac{\dg z_-}{|y_-|} d\e d\g,
	$$
	so that it can be controlled as in $G_{112}$. 
	
	Next we turn to $G_2$. To stay with the integration variables $\g$ and $\e$, we change $G_2(\g)$ to $G_2(\xi)$: 
	\begin{align*}
		G_2(\xi)=&-\int_{-\pi}^\xi \frac{\dg x}{|\dg x|^2}\cdot \dg \Big(\int\frac{\dg
			x_-}{|x_-|}d\e \Big)d\g+\int_{-\pi}^\xi \frac{\dg y}{|\dg y|^2}\cdot \dg \Big(\int\frac{\dg
			y_-}{|y_-|}d\e \Big)d\g=G_{21}+G_{22},
	\end{align*}
	where
	\begin{align*}
		&G_{21}=\int_{-\pi}^\xi \frac{\dg x}{|\dg x|^2}\cdot \int\dg
		x_-\frac{x_-\cdot\dg
			x_-}{|x_-|^3}d\e d\g-\int_{-\pi}^\xi \frac{\dg y}{|\dg y|^2}\cdot \int\dg
		y_-\frac{y_-\cdot\dg
			y_-}{|y_-|^3}d\e d\g,\\
		&G_{22}=-\int_{-\pi}^\xi \frac{\dg x}{|\dg x|^2}\cdot\int\frac{\dg^2
			x_-}{|x_-|}d\e d\g+\int_{-\pi}^\xi \frac{\dg y}{|\dg y|^2}\cdot \int\frac{\dg^2
			y_-}{|y_-|}d\e d\g.
	\end{align*}
	Recalling the identity \eqref{identity1}
	\begin{equation}\label{extraGP}
		x_-\cdot\dg x_-=(x_--\dg x\e)\cdot\dg x_-+\frac\e2|\dg x_-|^2,  
	\end{equation}
	we split $G_{21}=\sum_{j=1}^8G_{21j}$, where
	$$
	G_{211}=\int_{-\pi}^\xi \frac{\dg z}{|\dg x|^2}\cdot \int\dg
	x_-\frac{x_-\cdot\dg
		x_-}{|x_-|^3}d\e d\g,
	$$
	$$
	G_{212}=\int_{-\pi}^\xi \Big(\frac{1}{|\dg x|^2}-\frac{1}{|\dg y|^2}\Big)\dg y\cdot \int\dg
	x_-\frac{x_-\cdot\dg
		x_-}{|x_-|^3}d\e d\g,
	$$
	$$
	G_{213}=\int_{-\pi}^\xi \frac{\dg y}{|\dg y|^2}\cdot \int\dg
	z_-\frac{x_-\cdot\dg
		x_-}{|x_-|^3}d\e d\g,
	$$
	$$
	G_{214}=\int_{-\pi}^\xi \frac{\dg y}{|\dg y|^2}\cdot \int\dg
	y_-\Big(\frac{1}{|x_-|^3}-\frac{1}{|y_-|^3}\Big)x_-\cdot\dg
	x_-d\e d\g,
	$$
	$$
	G_{215}=\int_{-\pi}^\xi \frac{\dg y}{|\dg y|^2}\cdot \int\frac{\dg
		y_-}{|y_-|^3}(z_- -\dg z\e)\cdot\dg
	x_-d\e d\g,
	$$
	$$
	G_{216}=\int_{-\pi}^\xi \frac{\dg y}{|\dg y|^2}\cdot \int\frac{\dg
		y_-}{|y_-|^3}(y_- -\dg y\e)\cdot\dg
	z_-d\e d\g,
	$$
	$$
	G_{217}=\int_{-\pi}^\xi \frac{\dg y}{|\dg y|^2}\cdot \int\frac{\dg
		y_-}{|y_-|^3}\frac\e2\dg z_-\cdot\dg
	x_-d\e d\g,
	$$
	$$
	G_{218}=\int_{-\pi}^\xi \frac{\dg y}{|\dg y|^2}\cdot \int\frac{\dg
		y_-}{|y_-|^3}\frac\e2\dg y_-\cdot\dg
	z_-d\e d\g.
	$$
	Using \eqref{tvm} we  bound the first two as follows:
	$$
	|G_{211}|+|G_{212}|\leq C\|\dg x\|_{C^{\frac12}}\int_0^1dr\int \frac{d\e}{|\e|^{\mez}}\int d\g |\dg z||\dg^2 x(\g-r\e)|\leq C\|z\|_{H^1}.
	$$
	Combining \eqref{tvm} and \eqref{extraGP} yields
	$$
	|G_{2,3}|\leq C\|\dg x\|_{C^{\frac12}}\int_0^1dr\int \frac{d\e}{|\e|^{\mez}}\int d\g (|\dg z|+|\dg z'|)|\dg^2 x(\g-r\e)|\leq C\|z\|_{H^1}.
	$$
	Using \eqref{tvm} twice allows us to get
	$$
	|G_{214}|\leq C\|\dg y\|_{C^{\frac12}}\int_0^1dr\int_0^1d\tilde{r}\int \frac{d\e}{|\e|^{\mez}}\int d\g |\dg z(\g-r\e)||\dg^2 x(\g-\tilde{r}\e)|\leq C\|z\|_{H^1}.
	$$
	Similarly in $G_{215}$, we obtain
	$$
	|G_{215}|\leq C\|\dg y\|_{C^{\frac12}}\int_0^1\!\!\!dr\int_0^1\!\!\!d\tilde{r}\int \frac{d\e}{|\e|^{\mez}}\int d\g (|\dg z(\g\!-\!r\e)|\!+\!|\dg z|)|\dg^2 x(\g\!-\!\tilde{r}\e)|\leq C\|z\|_{H^1}.
	$$
	The next term is controlled as 
	$$
	|G_{216}|\leq C\|\dg y\|_{C^{\frac12}}\int_0^1dr\int \frac{d\e}{|\e|^{\mez}}\int d\g |\dg^2 y(\g-r\e)| (|\dg z|+|\dg z'|)\leq C\|z\|_{H^1}.
	$$ 
	The last two are estimated analogously: 
	\begin{align*}
		|G_{217}|+|G_{218}|&\leq C(\|\dg x\|_{C^{\frac12}}\!+\!\|\dg y\|_{C^{\frac12}})\int_0^1\!\!\!dr\!\int\! \frac{d\e}{|\e|^{\mez}}\int\! d\g |\dg^2 y(\g\!-\!r\e)| (|\dg z|\!+\!|\dg z'|)\\
		&\leq C\|z\|_{H^1}.
	\end{align*}
	We are left with the most singular term $G_{22}$. The identities
	$$
	\dg x\cdot\dg^2x_-=-\dg x\cdot\dg^2x'=-\dg x_-\cdot\dg^2x'
	$$
	yields
	\begin{align*}
		G_{22}=&\int_{-\pi}^\xi\int\frac{\dg x_-\cdot\dg^2
			x'}{|\dg x|^2|x_-|}d\e d\g-\int_{-\pi}^\xi\int
		\frac{\dg y_-\cdot\dg^2 y'}{|\dg y|^2|y_-|}
		d\e d\g.
	\end{align*}
	Hence, we can decompose  $G_{22}=\sum_{j=1}^4 G_{22j}$, where
	$$
	G_{221}=\int_{-\pi}^\xi\int\frac{\dg z_-\cdot\dg^2
		x'}{|\dg x|^2|x_-|}d\e d\g,\quad G_{222}=\int_{-\pi}^\xi\int\frac{\dg y_-\cdot\dg^2
		z'}{|\dg x|^2|x_-|}d\e d\g,
	$$
	$$
	G_{223}=\int_{-\pi}^\xi\int\frac{\dg y_-\cdot\dg^2
		y'}{|x_-|}\Big(\frac1{|\dg x|^2}-\frac1{|\dg y|^2}\Big)d\e d\g,
	$$
	$$
	G_{224}=\int_{-\pi}^\xi\int\frac{\dg y_-\cdot\dg^2
		y'}{|\dg y|^2}\Big(\frac1{|x_-|}-\frac1{|y_-|}\Big)d\e d\g.
	$$
	The more singular $G_{221}$ can be rewritten as 
	\begin{align*}
	G_{221}&=\int_{-\pi}^\xi\Big[\dg\Big(\int\frac{z_-\cdot\dg^2
		x'}{|\dg x|^2|x_-|}d\e\Big)-\int z_-\cdot\dg\Big(\frac{\dg^2
		x'}{|\dg x|^2|x_-|}\Big)d\e\Big] d\g\\
&=G_{221}^1+G_{221}^2+G_{221}^3,
\end{align*} where
	$$
	G_{221}^1=\Big(\int\frac{z_-\cdot\dg^2
		x'}{|\dg x|^2|x_-|}d\e\Big)\Big|_{\gamma=-\pi}^{\gamma=\xi},\quad G_{221}^2=-\int_{-\pi}^\xi\int z_-\cdot\frac{\dg^3
		x'}{|\dg x|^2|x_-|}d\e d\g,
	$$
	and 
	$$
	G_{221}^3=\int_{-\pi}^\xi\int z_-\cdot\frac{\dg^2
		x'}{|\dg x|^2}\frac{x_-\cdot\dg x_-}{|x_-|^3}d\e d\g.
	$$
	Identity \eqref{tvm} together with integration in the variable $\eta$ provides 
	\begin{align*}
		|G_{221}^1|\leq&2\|F(x)\|^3_{L^\infty}\int_0^1\frac{dr}{r^{\mez}}\|\dg z\|_{L^{2}}\|\dg^2
		x\|_{L^2}\leq  C\|z\|_{H^1}.
	\end{align*}
	Concerning $G_{221}^2$, the fact that
	$
	\dg^3 x'=\de\dg^2 x_-
	$ and integration by parts in $\eta$ yield
	$$
	G_{221}^2=-\int_{-\pi}^\xi\int z_-\cdot\dg^2 x_-\frac{ x_-\cdot\dg x' }{|\dg x|^2|x_-|^3}d\e d\g
	+\int_{-\pi}^\xi\int \dg z'\cdot\frac{\dg^2 x_-}{|\dg x|^2|x_-|}d\e d\g.
	$$
	Analogously to $G_{112}$, we get
	\begin{align*}
		|G_{221}^2|&\leq 2\|F(x)\|^3_{L^\infty}\int_0^1dr\int\frac{d\e}{|\eta|}\int d\g
		(|\dg z(\g-r\eta)|+|\dg z'|)|\dg^2 x_-|\\
		&\leq C\|\dg z\|_{L^2}\|x\|_{H^{2+s}} \leq C\|z\|_{H^1}.
	\end{align*}
	The desired bound  for $G_{221}$ is obtained. Regarding $G_{222}$, we similarly write 
	\begin{align*}
	G_{222}&=\int_{-\pi}^\xi\Big[\dg\Big(\int\frac{\dg y_-\cdot\dg
		z'}{|\dg x|^2|x_-|}d\e\Big)-\int \dg\Big(\frac{\dg
		y_-}{|\dg x|^2|x_-|}\Big)\cdot\dg z' d\e\Big] d\g
	\end{align*}
	and proceed similarly. The terms $G_{223}$ and $G_{224}$ are easier to control and we omit further details. 

	\subsection{$H^1$ estimates}
	Our goal is to prove that 
	\begin{equation}\label{zH1}
		\frac{d}{dt}\|\p_\g z\|_{H^1}^2\leq  C(\big||\dg x|-|\dg y|\big|\|z\|_{H^1}+\|z\|_{H^1}^2).
	\end{equation}
	We have 
	$$
	\frac{1}{2}\frac{d}{dt}\|\dg z\|^2_{L^2}=\int \dg z\cdot \dg z_t d\g=I_3+I_4,
	$$
	where
	$$
	I_3=\int \dg z\cdot \int \dg\Big(\frac{\dg x_-}{|x_-|}-\frac{\dg y_-}{|y_-|} \Big)d\e d\g,\quad
	I_4=\int \dg z\cdot \dg(\lambda(x)\dg x-\lambda(y)\dg y) d\g.
	$$
	We split further $I_3=I_{31}+I_{32}$, where 
	$$
	I_{31}=\int \dg z\cdot \int \Big(\frac{\dg^2 x_-}{|x_-|}-\frac{\dg^2 y_-}{|y_-|} \Big)d\e d\g,
	$$
	$$
	I_{32}=\int \dg z\cdot \int \Big(-\frac{\dg x_-(x_-\cdot \dg x_-)}{|x_-|^3}+\frac{\dg y_-(y_-\cdot \dg y_-)}{|y_-|^3} \Big)d\e d\g.
	$$
	Furthermore, $I_{31}=I_{311}+I_{312}$, where
	$$I_{311}=\int \dg z\cdot \int \frac{\dg^2 z_-}{|x_-|} d\e d\g,
	\quad
	I_{312}=\int \dg z\cdot \int \dg^2 y_-\Big(\frac{1}{|x_-|}-\frac{1}{|y_-|}\Big) d\e d\g.
	$$
	We show that all the induced terms are bounded by the right-hand side of \eqref{zH1}. 
	
	Symmetrizing and using \eqref{extraGP} yield
	\begin{align}
		\begin{split}\label{I311}
			|I_{311}|&=\left|\frac14\int\int |\dg z_-|^2  \frac{x_-\cdot\dg x_-}{|x_-|^3} d\e d\g\right|\\
			&\leq 
			\|F(x)\|_{L^\infty}^3\|\dg x\|_{C^{\frac12+s}}^2\int\frac{d\e}{|\e|^{1-2s}}\int d\g(|\dg z|^2+|\dg z'|^2)\leq C\|\dg z\|^2_{L^2}.
		\end{split}
	\end{align}
	 $I_{312}$ is the most difficult term  due to the low regularity. 
	 We perform a further splitting, 
	\begin{equation}\label{SI312}
		I_{312}=J_1+J_2+J_3,
	\end{equation} 
	\begin{align*}
	&J_1=\int \dg z\cdot\dg^2 y \int \Big(\frac{1}{|x_-|}-\frac{1}{|\dg x||\e|}-\Big(\frac{1}{|y_-|}-\frac{1}{|\dg y||\e|}\Big)\Big) d\e d\g,\\
	&J_2=-\int \dg z\cdot \int \dg^2 y'\Big(\frac{1}{|x_-|}-\frac{1}{|\dg x||\e|}-\Big(\frac{1}{|y_-|}-\frac{1}{|\dg y||\e|}\Big)\Big) d\e d\g,\\
	&J_3=\int \dg z\cdot \int \dg^2 y_-\Big(\frac{1}{|\dg x||\e|}-\frac{1}{|\dg y||\e|}\Big) d\e d\g.
	\end{align*}
	Using that the modulus of the tangent vectors only depend on time, we have $|\p_\g x|=|\p_\g x'|$ and thus 
	\begin{align*}
		\frac{1}{|x_-|}-\frac{1}{|\dg x||\e|}&=\frac{|\dg x'||\e|-|x_-|}{|x_-||\dg x'||\e|}
		=\frac{(\dg x'\e-x_-)\cdot(\dg x'\e+x_-)}{|x_-||\dg x'||\e|(|x_-|+|\dg x'||\e|)}\\
		&=\frac{2(\dg x'\e-x_-)\cdot\dg x'\e}{|x_-||\dg x'||\e|(|x_-|+|\dg x'||\e|)}-
		\frac{|\dg x'\e-x_-|^2}{|x_-||\dg x'||\e|(|x_-|+|\dg x'||\e|)}.
	\end{align*}
	The same decomposition holds for $y$, implying $$J_1=J_{11}+J_{12}$$ with
	\begin{align*}
	&J_{11}=\int \dg z\cdot\dg^2 y \int \Big(\frac{2(\dg x'\e-x_-)\cdot\dg x'\e}{|x_-||\dg x'||\e|(|x_-|\!+\!|\dg x'||\e|)}-\frac{2(\dg y'\e-y_-)\cdot\dg y'\e}{|y_-||\dg y'||\e|(|y_-|\!+\!|\dg y'||\e|)}\Big) d\e d\g,\\
	&J_{12}=\int \dg z\cdot\dg^2 y \int \Big(\frac{|\dg y'\e-y_-|^2}{|y_-||\dg y'||\e|(|y_-|\!+\!|\dg y'||\e|)}-\frac{|\dg x'\e-x_-|^2}{|x_-||\dg x'||\e|(|x_-|\!+\!|\dg x'||\e|)}\Big) d\e d\g.
	\end{align*}
	Using \eqref{tvm} and that $\p_\g x\cdot \p_\g^2 x=0$ gives 
	\[
	2(\dg x'\e-x_-)\cdot\dg x'\e=2(\dg x'-\dg x(\g-r\e))\cdot\dg x'\e^2=|\dg x'-\dg x(\g-r\e)|^2\e^2,
	\]
	whence 
	\begin{align*}
		J_{11}=\int d\g \dg z\cdot\dg^2 y \int d\e \int_0^1 dr \Big(&\frac{|\dg x'-\dg x(\g-r\e)|^2\e^2}{|x_-||\dg x'||\e|(|x_-|\!+\!|\dg x'||\e|)} -\frac{|\dg y'-\dg y(\g-r\e)|^2\e^2}{|y_-||\dg y'||\e|(|y_-|\!+\!|\dg y'||\e|)}\Big).	
	\end{align*}
	Now,  $J_{11}=\sum_{j=1}^5J_{11j}$, where 
	\begin{align*}
		J_{111}=\int d\g \dg z\cdot\dg^2 y \int d\e \int_0^1 dr \frac{(\dg z'-\dg z(\g-r\e))\cdot(\dg x'-\dg x(\g-r\e))|\e|}{|x_-||\dg x'|(|x_-|\!+\!|\dg x'||\e|)},	
	\end{align*}
	\begin{align*}
		J_{112}=\int d\g \dg z\cdot\dg^2 y \int d\e \int_0^1 dr \frac{(\dg y'-\dg y(\g-r\e))\cdot(\dg z'-\dg z(\g-r\e))|\e|}{|x_-||\dg x'|(|x_-|\!+\!|\dg x'||\e|)},	
	\end{align*}
	\begin{align*}
		J_{113}=\int d\g \dg z\cdot\dg^2 y \int d\e \int_0^1 dr \frac{|\dg y'-\dg y(\g-r\e)|^2|\e|}{|\dg x'|(|x_-|\!+\!|\dg x'||\e|)}\Big(\frac1{|x_-|}-\frac1{|y_-|}\Big),	
	\end{align*}
	\begin{align*}
		J_{114}=\int d\g \dg z\cdot\dg^2 y \int d\e \int_0^1 dr \frac{|\dg y'-\dg y(\g-r\e)|^2|\e|}{|y_-|(|x_-|\!+\!|\dg x'||\e|)}\Big(\frac1{|\dg x'|}-\frac1{|\dg y'|}\Big),	
	\end{align*}
	and
	\begin{align*}
		J_{115}=\int d\g \dg z\cdot\dg^2 y \int d\e \int_0^1 dr& \frac{|\dg y'-\dg y(\g-r\e)|^2|\e|}{|y_-||\dg y'|}\\
		&\qquad \times \Big(\frac1{|x_-|\!+\!|\dg x'||\e|}-\frac1{|y_-|\!+\!|\dg y'||\e|}\Big).	
	\end{align*}
	In $J_{111}$ it is possible to use the next mean value identity
	$$
	\dg x'-\dg x(\g-r\e)=\e(r-1)\int_0^1d\rho \dg^2x(\gamma-(\rho+r(1-\rho))\eta)
	$$ 
	to bound as follows
	\begin{align*}
		|J_{111}|\leq\|F(x)\|^3_{L^\infty}\int d\g& |\dg z||\dg^2 y|\int_0^1 (1-r)dr\int_0^1 d\rho \\
		&\times \int d\e  (|\dg z'|+|\dg z(\g-r\e)|)|\dg^2x(\gamma-(\rho+r(1-\rho))\eta)|.
	\end{align*}
	Integrating first in $\eta$ and later in $\gamma$ it is possible to get the desired control:
	\begin{align*}
		J_{111}&\leq C\int_0^1 \frac{1-r}{r^{\mez}}dr\int_0^1 \frac{d\rho}{(\rho+r(1-\rho))^{\mez}} \|\dg z\|_{L^2}^2\|\dg^2 y\|_{L^2}\|\dg^2 x\|_{L^2}\\
		&\leq C\int_0^1 \frac{dr}{r^{\mez}}\int_0^1 \frac{d\rho}{\rho^{\mez}} \|\dg z\|_{L^2}^2\leq C\|z\|_{H^1}^2.
	\end{align*}
The other $J_{11j}$ and $J_{1, 2}$ can be treated in the same fashion, giving the desired bound for $J_1$.  As for $J_2$, we only consider the following counterpart of $J_{111}$ as the control for other terms follows along the same lines.
\begin{align*}
		J_{211}=-\int d\g \dg z\cdot \int d\e \dg^2 y' \int_0^1 dr \frac{(\dg z-\dg z(\g-r\e))\cdot(\dg y-\dg y(\g-r\e))|\e|}{|y_-||\dg y|(|y_-|\!+\!|\dg y||\e|)}.
\end{align*}
The mean value theorem gives
	$$
	\dg y-\dg y(\g-r\e)=r\e\int_0^1d\rho\dg^2y(\gamma+(\rho-1)r\e),
	$$ 
	so that
	\begin{align*}
		|J_{211}|\leq& C\int_0^1\!\! rdr\int_0^1\!\!d\rho \int\! d\g |\dg z| \int\! d\e |\dg^2 y'|(|\dg z|+|\dg z(\g\!-\! r\e)|)|\dg^2 y(\g+(\rho\!-\!1)r\e)|=K_1+K_2,
	\end{align*}
	where
	\begin{align*}
	&K_1= C\int_0^1\!\! rdr\int_0^1\!\!d\rho \int\! d\g |\dg z|^2 \int\! d\e |\dg^2 y'||\dg^2 y(\g+(\rho-1)r\e)|,\\
	&K_2= C\int_0^1\!\! rdr\int_0^1\!\!d\rho \int\int d\g d\e |\dg z||\dg^2 y'||\dg z(\g-r\e)|)|\dg^2 y(\g+(\rho-1)r\e)|.
	\end{align*}
	By Cauchy-Schwarz's inequality in $\eta$, 
	\begin{align*}
		K_{1}\leq& C\int_0^1\!\! r^{\mez}dr\int_0^1\!\!\frac{d\rho}{(1-\rho)^{\mez}}\|z\|_{H^1}^2\leq C\|z\|_{H^1}^2.
	\end{align*}
	For $K_2$ we use Cauchy-Schwarz's inequality in both $\g$ and $\e$ to have
	\begin{align*}
		K_2&\leq C\int_0^1\!\! rdr\!\!\int_0^1\!\!d\rho \Big(\int\!\!\!\int\!\! d\g d\e |\dg z|^2|\dg^2 y(\g\!+\!(\rho\!-\!1)r\e)|^2\Big)^{\frac12}\Big(\int\!\!\!\int\!\! d\g d\e|\dg^2 y'|^2|\dg z(\g\!-\!r\e)|^2\Big)^{\frac12}\\
		&\leq C\int_0^1\!\! \frac{r^{\mez}dr}{(1-r)^{\mez}}\!\!\int_0^1\!\!\frac{d\rho}{(1-\rho)^{\mez}} \|\dg z\|_{L^2}^2\|\dg^2 y\|_{L^2}^2\leq C\|z\|_{H^1}^2,
	\end{align*}
	where we have made the change of variables $\zeta=\g-\eta$, $d\g=d\zeta$ in the second double integral. Consequently, $|J_{211}|\leq C\|z\|_{H^1}^2$ and thus 
	\[
	|J_1|+|J_2|\le  C\|z\|_{H^1}^2.
	\]
	To bound $J_3$ we recall that $|\p_\g x|$ and $|\p_\g y|$ depend only on $t$, so that as in the control of $G_{112}$, 
	\begin{align*}
		|J_3|&\leq \| F(x)\|_{L^\infty} \| F(y)\|_{L^\infty}\big||\dg x|-|\dg y|\big| \int\frac{d\e}{|\e|}\int d\g |\dg z||\dg^2y_-|\\
		&\les \| F(x)\|_{L^\infty} \| F(y)\|_{L^\infty}\big||\dg x|-|\dg y|\big|\| \p_\g z\|_{L^2}\| \dg^2y\|_{\dot H^s}\\
		&\le C\big||\dg x|-|\dg y|\big|\|\dg z\|_{L^2}.
	\end{align*}
	We have proved that $I_{312}=J_1+J_2+J_3$ is bounded by
	$$
	|I_{312}|\leq C(\big||\dg x|-|\dg y|\big|\|z\|_{H^1}+\|z\|_{H^1}^2)
	$$
	which in conjunction with \eqref{I311} yields
	\begin{equation}\label{I31}
		|I_{31}|\leq C(\big||\dg x|-|\dg y|\big|\|z\|_{H^1}+\|z\|_{H^1}^2).
	\end{equation}
	Note that the control quantity $\big||\dg x|-|\dg y|\big|\|z\|_{H^1}$ is {\it only due to $J_3$}. 
	
	Regarding $I_{32}$, we use identity \eqref{extraGP} to split $I_{32}=\sum_{j=1}^8I_{32j}$, where
	$$
	I_{321}=-\int \dg z\cdot \int \dg z_- \frac{(x_--\dg x\e)\cdot \dg x_-}{|x_-|^3} d\e d\g,$$
	
	$$
	I_{322}=-\int \dg z\cdot \int\dg y_- \frac{(z_--\dg z\e)\cdot \dg x_-}{|x_-|^3} d\e d\g,
	$$
	$$
	I_{323}=-\int \dg z\cdot \int\dg y_- \frac{(y_--\dg y_-\eta)\cdot \dg z_-}{|y_-|^3} d\e d\g,
	$$
	$$
	I_{324}=-\int \dg z\cdot \int \dg y_-(y_--\dg y_-\e)\cdot \dg y_-(|x_-|^{-3}-|y_-|^{-3}) d\e d\g,
	$$
	$$
	I_{325}=-\int \dg z\cdot \int \dg z_- \frac{\e|\dg x_-|^2}{2|x_-|^3} d\e d\g,
	$$
	$$
	I_{326}=-\int \dg z\cdot \int\dg y_- \frac{\e \dg z_-\cdot \dg x_-}{2|x_-|^3} d\e d\g,
	$$
	$$
	I_{327}=-\int \dg z\cdot \int\dg y_- \frac{\e \dg y_-\cdot \dg z_-}{2|x_-|^3} d\e d\g,
	$$
	and
	$$
	I_{328}=-\frac12\int \dg z\cdot \int \dg y_-\e |\dg y_-|^2(|x_-|^{-3}-|y_-|^{-3}) d\e d\g.
	$$
	Analogously to the control of $G_{112}$ we find $I_{32}\leq C\|z\|_{H^1}^2$. Combining this with \eqref{I31} we obtain
	\begin{equation}\label{I3}
		|I_{3}|\leq C(\big||\dg x|-|\dg y|\big|\|z\|_{H^1}+\|z\|_{H^1}^2).
	\end{equation}
	Finally, we write $I_4=\sum_{j=1}^4I_{4j}$, where
	\begin{align*}
	&I_{41}=\int \dg z\cdot \lambda(x)\dg^2 z d\g,\quad
	I_{42}=\int \dg z\cdot (\lambda(x)-\lambda(y))\dg^2 y d\g,\\
	&I_{43}=\int |\dg z|^2\dg\lambda(x) d\g,\quad
	I_{44}=\int \dg z\cdot\dg y (\dg\lambda(x)-\dg\lambda(y)) d\g.
	\end{align*}
	Integration by parts in $I_{41}$ yields $I_{41}=-\mez I_{43}$, while $$|I_{43}|\leq \|\dg z\|_{L^2}^2\|\dg\lambda(x)\|_{L^\infty}\leq C\|z\|^2_{H^1},$$
	where we have used. By virtue of \eqref{ld:C1},  we have 
	$$|I_{42}|\leq C \|\dg^2 y\|_{L^2}\|\dg z\|_{L^2}\|\lambda(x)-\lambda(y)\|_{L^\infty}\leq C\|z\|^2_{H^1}.$$
       Next, we integrate by parts
	$$
	I_{44}=-\int \dg z\cdot\dg^2 y (\lambda(x)-\lambda(y)) d\g-\int \dg^2 z\cdot\dg y (\lambda(x)-\lambda(y)) d\g
	$$
	and use the identity $\dg^2 z\cdot\dg y=-\dg^2 x\cdot\dg z$ to have
	$$
	I_{44}\leq (\|\dg^2 y\|_{L^2}+\|\dg^2 x\|_{L^2})\|\dg z\|_{L^2} \|\lambda(x)-\lambda(y)\|_{L^\infty}\leq C\|z\|_{H^1}^2.
	$$
	It follows that
	\bq\label{est:I4}
	I_4\leq C\|z\|_{H^1}^2.
	\eq
	 A  combination of \eqref{I3} and \eqref{est:I4} leads to \eqref{zH1}.
	\subsection{Estimates for $\big||\dg x|-|\dg y|\big|$}
	 Identity
	$$
	\frac{d}{dt}|\dg x|^2=\frac1\pi\int\dg x\cdot\dg\Big(\int\frac{\dg x_-}{|x_-|}d\e\Big)d\g,
	$$
	and integration by parts provide
	$$
	\frac{d}{dt}|\dg x|=\frac{-1}{2\pi}\int\frac{\dg^2 x}{|\dg x|}\cdot\int\frac{\dg x_-}{|x_-|}d\e d\g.
	$$
	This new identity allows to find the following spiting
	$$
	\frac{d}{dt}(|\dg x|-|\dg y|)=I_5+I_6+I_7+I_8,
	$$
	where 
	$$
	I_5=\frac{-1}{2\pi}\int \big(\frac{1}{|\dg x}-\frac{1}{|\dg y|}\big)\dg^2 x\cdot\int\frac{\dg x_-}{|x_-|}d\e d\g,
	\quad
	I_6=\frac{-1}{2\pi}\int\frac{\dg^2 z}{|\dg y|}\cdot\int\frac{\dg x_-}{|x_-|}d\e d\g,
	$$
	$$
	I_7=\frac{-1}{2\pi}\int\frac{\dg^2 y}{|\dg y|}\cdot\int\frac{\dg z_-}{|x_-|}d\e d\g,
	\quad\mbox{and}\quad
	I_8=\frac{-1}{2\pi}\int\frac{\dg^2 y}{|\dg y|}\cdot\int\dg y_-\big(\frac{1}{|x_-|}-\frac{1}{|y_-|})d\e d\g.
	$$
	Proceeding as before, we can obtain
	$$
	I_5\leq \|F(x)\|_{L^\infty}^3\|\dg x\|_{C^{\frac12}}\int\frac{d\e}{|\e|^{\mez}}\int d\g|\dg^2 x||\dg z|\leq C\|z\|_{H^1}.
	$$ 
	It is possible to integrate by parts in $I_6$ in such a way that
	$$
	I_6=\frac{1}{2\pi}\int\frac{\dg z}{|\dg y|}\cdot\int\frac{\dg^2 x_-}{|x_-|}d\e d\g-\frac{1}{2\pi}\int\frac{\dg z}{|\dg y|}\cdot\int\dg x_-\frac{x_-\cdot\dg x_-}{|x_-|^3}d\e d\g.
	$$
	Therefore,  following the control of $G_{1,1,2}$ (see \eqref{G112}) we bound
	\begin{align*}
		I_6&\les \|F(y)\|_{L^\infty}\left\{\|F(x)\|_{L^\infty}\|\dg z\|_{L^2}\|\p_\g^2 x\|_{\dot H^{s}}+\|F(x)\|_{L^\infty}^2\|\dg x\|_{C^{\frac12}}\|\dg z\|_{L^2}\|\dg^2x\|_{L^{2}}\right\}\\
		&\leq C\|z\|_{H^1}.
	\end{align*}
	Next, we symmetrize $I_7$ as
	$$
	I_7=\frac{-1}{4\pi}\int\!\!\int\frac{\dg^2 y_-}{|\dg y|}\cdot\frac{\dg z_-}{|x_-|}d\e d\g=\frac{-1}{2\pi}\int\frac{\dg z}{|\dg y|}\cdot\int\frac{\dg^2 y_-}{|x_-|}d\e d\g,
	$$
	so that
	$$
	I_7\leq \|F(y)\|_{L^\infty}\|F(x)\|_{L^\infty}\|\dg z\|_{L^2}\|\p_\g^2y\|_{\dot H^s}\leq C\|z\|_{H^1}.
	$$
	Finally,
	\begin{align*}
	I_8&\les  \|F(y)\|_{L^\infty}^2\|F(x)\|_{L^\infty}\|\dg y\|_{C^{\frac12}}\int_0^1\int\int|\p_\g^2y(\g)||\p_\g z(\g-r\eta)|d\g\frac{d\eta}{|\eta|^\mez}dr\\
	&\les \|F(y)\|_{L^\infty}^2\|F(x)\|_{L^\infty}\|\dg y\|_{C^{\frac12}}\|\dg^2y\|_{L^{2}}\|\dg z\|_{L^2} \leq C\|z\|_{H^1}.
	\end{align*}
	We have proved that
	$$
	\big|\frac{d}{dt}(|\dg x|-|\dg y|)\big|\leq C\|z\|_{H^1}
	$$
	which together with \eqref{zH1} implies the closed differential inequality 
	\bq
	\frac{d}{dt}\big(\|z\|^2_{H^1}+\big||\dg x|-|\dg y|\big|^2\big)\leq  C\big(\|z\|^2_{H^1}+\big||\dg x|-|\dg y|\big|^2\big).
	\eq
	The use of Gronwall's inequality yields the stability estimate \eqref{stability}, finishing the proof of Theorem \ref{theo:stability}.
	\section{Proof of the main result--Theorem \ref{mainTheorem}}\label{sec:existence}

In this section we construct solutions of system \eqref{SQGp}-\eqref{ld}  with initial data $x^o(\g)\in H^{2+s}$, $s\in (0, \mez)$. The uniqueness has been proved in Theorem \ref{theo:stability}. The existence is done by regularizing the initial data and applying the existence result for regular ($H^3$) data established in \cite{G}. 
Since our a priori estimates crucially use the fact that the tangent vector's length is independent of the parameter of the curve, the regularization procedure must maintain this property. Assume that the initial closed curve $x^o(\gamma)\in H^{2+s}$ satisfies the arc chord condition, i.e. $\|F(x^o)\|_{L^\infty}<\infty$. In addition, upon reparametrizing (see below) we may assume $\dg |\dg x^o(\g)|=0$. 
First, we mollify $x^o$ to get 
$$
x^o_\varepsilon(\gamma)=(\Gamma_\varepsilon*x^o)(\g),
$$
with $\Gamma_\varepsilon$ an approximation of the identity. It is clear that $x^o_\varepsilon\in C^{\infty}(\T)$ and
\begin{equation}\label{xoepsproperties}
	\lim_{\eps \to 0}\|x^o_\varepsilon-x^o\|_{H^{2+s}}= 0.
\end{equation}
This ensures that for small $\eps$, $x^o_\eps$ satisfies the arc chord condition. To obtain the constant length (in parameter) of the tangent vector, we reparametrize using
$$
\phi_\varepsilon:[-\pi,\pi]\to[-\pi,\pi],\quad \phi_{\varepsilon}(\xi)=-\pi+\frac{2\pi}{L_\varepsilon}\int_{-\pi}^{\xi}|\dg  x^o_\varepsilon(\g)|d\g,\quad L_\varepsilon=\int_{-\pi}^\pi|\dg  x^o_\varepsilon(\g)|d\g.
$$
It is clear that $\phi_\varepsilon\in C^{\infty}(\T)$ and its inverse  $\phi_\varepsilon^{-1}$ is well defined. Then the curve
$$
\tilde{x}^o_\varepsilon(\gamma)=x^o_\varepsilon(\phi_\varepsilon^{-1}(\g))
$$
satisfies 
$$
\tilde{x}^o_\varepsilon\in C^\infty(\T),\quad\dg |\dg \tilde{x}^o_\varepsilon(\g)|=0.
$$
(In fact, $ |\dg \tilde{x}^o_\varepsilon(\g)|=\frac{L_\eps}{2\pi}$). We postpone the proof of the following lemma to the end of this section. 
\begin{lemm}\label{lemm:data}
\bq\label{converge:data}
 \lim_{\eps\to 0}\|\tilde{x}^o_\varepsilon-x^o\|_{H^{2+s}}=0.
\eq
\end{lemm}
Now  $\tilde{x}^o_\varepsilon$ is a smooth initial curve satisfying the arc chord condition and $\dg |\dg \tilde{x}^o_\varepsilon(\g)|=0$. Set $\eps=n^{-1}$ and relabel $ \tilde{x}^o_\varepsilon= \tilde{x}^o_n$. Applying the existence result in \cite{G} we obtain for each $n$ a solution $  x_n\in C([0, T_n]; H^3)$ satisfying the arc chord condition and $  x_n\vert_{t=0}=\tilde{x}^o_n$.  The convergence \eqref{converge:data} ensures that the sequences $\| \tilde{x}^0_n\|_{H^{2+s}}$ and $\| F(\tilde{x}^o_n)\|_{L^\infty(\T\times \T)}$ are bounded. Using this and a continuity argument, we deduce from the a priori estimates \eqref{est:norm} and \eqref{est:F} that there exists $0<T<T_n$ for all $n$ such that the sequences  
\[
\|  x_n\|_{C([0, T]; H^{2+s})}\quad\text{and}\quad \| F(  x_n)\|_{L^\infty([0, T]; L^\infty(\T\times \T))}
\]
 are bounded. It follows easily from equations \eqref{SQGp}-\eqref{ld} with the aid of \eqref{dld:Linfty} that  $\p_t  x_n$ is uniformly bounded in $L^\infty([0, T]; L^2)$. By virtue of the Aubin-Lions lemma, there exists $x\in C_w([0, T]; H^{2+s})\cap C([0, T]; H^2)$ such that (upon extracting a subsequence)
 \bq\label{convergence:xn}
  x_n\wsc x \quad \text{in } L^\infty([0, T]; H^{2+s})\quad \text{and}\quad  x_n\to x\quad\text{in } C([0, T]; H^2).
 \eq
Using these convergences it can be shown that $x$ is a solution of  \eqref{SQGp}-\eqref{ld}. We note that $\ld(x_n)\to \ld(x)$ in $L^\infty$ in view of \eqref{boundsLambdaLinfty}. Next, we show the continuity in time $x\in C([0, T]; H^{2+s})$. Indeed, from \eqref{est:norm}  $x_n$ satisfies 
\[
\| x_n(t)\|_{H^{2+s}}^2\le \| \tilde{x}^o_n\|_{H^{2+s}}^2\exp\Big(\int_0^t  \cP(\| F(x_n)(t')\|_{L^\infty}, \| \p_\g^2x_n(t')\|_{H^s})dt'\big)\le  \| \tilde{x}^o_n\|_{H^{2+s}}^2\exp(tC(T))
\]
for all $t\in [0, T]$. Letting $n\to \infty$ yields 
\[
\| x\|_{L^\infty([0, t]; H^{2+s})}^2\le \limsup_{n\to \infty}\|  x_n\|_{L^\infty([0, t]; H^{2+s})}^2\le \| x^o\|_{H^{2+s}}^2\exp(tC(T)),\quad t\in [0, T].
\]
It follows that 
\[
\lim_{t\to 0^+}\| x(t)\|_{H^{2+s}}\le  \| x^o\|_{H^{2+s}}
\]
which combined with the weak continuity $x\in C_w([0, T]; H^{2+s})$ implies that $x$ is continuous from the right at $t=0$ with values in $H^{2+s}$. Next, for any $t_0\in (0,  T)$, we consider $x(t_0)$ as the new initial data. The above argument gives a solution  $\wt x\in L^\infty([t_0, t_0+\delta]; H^{2+s})$ with $\delta=\delta(t_0) \in (0, T-t_0)$ such that $\wt x$ is continuous from the right at $t_0$ with values in $H^{s+2}$. The same property holds for $x$ since the uniqueness result in Theorem \ref{theo:stability} implies $x=\wt x$ on $[t_0, t_0+\delta]$. The left continuity can be obtained similarly using the fact that \eqref{SQGp}-\eqref{ld} are time reversible.

{\it Proof of Lemma \ref{lemm:data}}

We shall prove that
\begin{equation}\label{phiproperties}
	\lim_{\eps \to 0} \|\phi_\varepsilon(\,\cdot\,)-\cdot\,\|_{H^{2+s}}= 0
\end{equation}
which in turn yields same for the inverse
$$
\lim_{\eps\to 0} \|\phi_\varepsilon^{-1}(\,\cdot\,)-\cdot\,\|_{H^{2+s}}= 0
$$	
upon using the formula for derivatives of inverse functions. These imply the desired convergence
$$
\|\tilde{x}^o_\varepsilon-x^o\|_{H^{2+s}}\leq \|x^o_\varepsilon\circ\phi^{-1}_\varepsilon-x^o_\varepsilon\|_{H^{2+s}}+\|x^o_\varepsilon-x^o\|_{H^{2+s}}\to 0\quad\mbox{as}\quad \varepsilon\to0.
$$
It can be easily checked that $\| \phi_\eps(\cdot)-\cdot\|_{L^2}\to 0$. For the highest order derivative, we first compute 
 $$
\partial_\g^2\phi_\varepsilon(\g)=\frac{2\pi}{L_\varepsilon}\frac{\dg x^o_\varepsilon(\g)\cdot \dg^2x^o_\varepsilon(\g)}{|\dg x^o_\varepsilon(\g)|}.
$$
Then using the formula 
$$
\Lambda^{s}\partial_\g^2\phi_\varepsilon(\g)=c\int_{\Rr}\frac{\partial_\g^2\phi_\varepsilon(\g)-\partial_\g^2\phi_\varepsilon(\g-\e)}{|\e|^{1+s}}d\e
$$
and  the identity
$$
\Lambda^{s}\big(\frac{2\pi}{L}\frac{\dg x^o(\g)\cdot \dg^2x^o(\g)}{|\dg x^o(\g)|}\big)=\Lambda^{s}(0)=0,\quad L=\int_{-\pi}^\pi|\dg x^o(\g)|d\g=2\pi|\dg x^o(\g)|,
$$
we can decompose 
$$
\Lambda^{s}\partial_\g^2\phi_\varepsilon=\sum_{j=1}^6G_j,
$$
where
$$
G_1(\gamma)=2\pi\Big(\frac{1}{L_\varepsilon}-\frac1{L}\Big)\frac{\dg x^o_\varepsilon(\g)}{|\dg x^o_\varepsilon(\g)|}\cdot\Lambda^s\dg x^o_\varepsilon(\gamma),\quad G_2(\gamma)=\frac{2\pi}{L}\Big(\frac{\dg x^o_\varepsilon(\g)}{|\dg x^o_\varepsilon(\g)|}-\frac{\dg x^o(\g)}{|\dg x^o(\g)|}\Big)\cdot\Lambda^s\dg x^o_\varepsilon(\gamma),
$$
$$
G_3(\gamma)=\frac{2\pi}{L}\frac{\dg x^o(\g)}{|\dg x^o(\g)|}\cdot(\Lambda^s\dg x^o_\varepsilon-\Lambda^s\dg x^o),
$$
$$
G_4(\gamma)=\Big(\frac{2\pi c}{L_\varepsilon}-\frac{2\pi c}{L}\Big) \int_{\Rr}\Big(\frac{\dg x^o_\varepsilon(\g)}{|\dg x^o_\varepsilon(\g)|}-\frac{\dg x^o_\varepsilon(\g-\e)}{|\dg x^o_\varepsilon(\g-\e)|}\Big)\cdot\frac{\dg^2x^o_\varepsilon(\g-\e)}{|\e|^{1+s}}d\e,
$$
$$
G_5(\gamma)=\frac{2\pi c}{L} \int_{\Rr}\Big[\Big(\frac{\dg x^o_\varepsilon(\g)}{|\dg x^o_\varepsilon(\g)|}-\frac{\dg x^o(\g)}{|\dg x^o(\g)|}\Big)-\Big(\frac{\dg x^o_\varepsilon(\g-\e)}{|\dg x^o_\varepsilon(\g-\e)|}-\frac{\dg x^o(\g-\e)}{|\dg x^o(\g-\e)|}\Big)\Big]\cdot\frac{\dg^2x^o_\varepsilon(\g-\e)}{|\e|^{1+s}}d\e,
$$
and
$$
G_6(\gamma)=\frac{2\pi c}{L} \int_{\Rr}\Big(\frac{\dg x^o(\g)}{|\dg x^o(\g)|}-\frac{\dg x^o(\g-\e)}{|\dg x^o(\g-\e)|}\Big)\cdot\frac{\dg^2x^o_\varepsilon(\g-\e)-\dg^2x^o(\g-\e)}{|\e|^{1+s}}d\e.
$$
Using the convergence \eqref{xoepsproperties}, we can show that $\| G_j\|_{L^2}\to 0$ as $\eps\to 0$ for all $1\le j\le 6$. Indeed, denoting $B=|\dg x^o(\g)|>0$ independent of $\g$, \eqref{xoepsproperties} implies 
\[
\exists\, \eps_0>0,~  \forall \eps\in (0, \eps_0),~\forall \g\in \T,~B_\eps(\g):=|\dg x^o_\eps(\g)|\ge \frac{B}{2}.
\]
 Note in addition that $B^{-1}\le \| F(x^o)\|_{L^\infty}$ and $\| B_\eps-B\|_{L^\infty}\to 0$. In particular, $L_\eps\to L$ and thus $\| G_1\|_{L^2}\to 0$. As for $G_2$ we estimate 
\begin{align*}
\| G_2\|_{L^2}&\le \frac{2\pi}{L}\Big\| \frac{\dg x^o_\varepsilon(\cdot)}{|\dg x^o_\varepsilon(\cdot)|}-\frac{\dg x^o(\cdot)}{|\dg x^o(\cdot)|}\Big\|_{L^\infty}\|\Lambda^s\dg x^o_\varepsilon\|_{L^2},
\end{align*}
where the $L^\infty$ norm tends to $0$ because \eqref{xoepsproperties} implies that $\p_\g x^o_\varepsilon\to \p_\g x^o$ in $H^{1+s}\subset C^{\mez+s}$. The term $G_3$ is obvious. Since $L_\eps \to L$, in $G_4$ it suffices to bound 
\begin{align*}
K&:=\left\|\int_{\Rr}\Big(\frac{\dg x^o_\varepsilon(\cdot)}{|\dg x^o_\varepsilon(\cdot)|}-\frac{\dg x^o_\varepsilon(\cdot-\e)}{|\dg x^o_\varepsilon(\cdot-\e)|}\Big)\cdot\frac{\dg^2x^o_\varepsilon(\cdot-\e)}{|\e|^{1+s}}d\e\right\|_{L^2(\T)}\\
&\le \big\|\frac{\dg x^o_\varepsilon }{|\dg x^o_\varepsilon|}\big\|_{C^{\frac12}}\Big(\int_{-\pi}^\pi\Big(\int_{-1}^{1}\frac{|\dg^2x^o_\varepsilon(\g-\e)|}{|\e|^{\frac12+s}}d\e\Big)^2d\g\Big)^{\frac12}+ 2\Big(\int_{-\pi}^\pi\Big(\int_{|\eta|>1}\frac{|\dg^2x^o_\varepsilon(\g-\e)|}{|\e|^{1+s}}d\e\Big)^2d\g\Big)^{\frac12},
\end{align*}
where 
\[
\big\|\frac{\dg x^o_\varepsilon }{|\dg x^o_\varepsilon|}\big\|_{C^{\frac12}}\le M=M(\|\dg x^o_\varepsilon\|_{C^{\mez}}, B).
\]
 Minkowski's inequality  allows to bound 
\begin{align*}
K&\le\Big(M\int_{-1}^1\frac{d\e}{|\e|^{\frac12+s}}+2\int_{|\eta|>1}\frac{d\e}{|\e|^{1+s}}\Big)\Big(\int_{-\pi}^\pi|\dg^2x^o_\varepsilon(\g-\e)|^2d\g\Big)^{\frac12}\le C(M+1) \| \p_\g^2 x^o_\varepsilon\|_{L^2}
\end{align*}
which is uniformly bounded in $\eps$.  Here and in the remainder of this proof, $C$ denotes absolute  constants. By an analogous argument using that $B$ is independent of $\g$, we obtain 
\[
\| G_6\|_{L^2}\le \frac{C}{LB}\|\p_\g x^o\|_{C^\mez}\| \p_\g^2 x^o_\varepsilon-\p_\g^2 x^o\|_{L^2}\to 0.
\]
As for $G_5$, using the notation $x_-(\g, \eta)=x(\g)-x(\g-\eta)$ we rewrite 
\begin{align*}
G_5(\gamma)
&=\frac{2\pi c}{L} \int_{\Rr}\Big[\frac{(\p_\g x^o_\eps-\p_\g x^o)_-(\g, \eta)}{B_\eps(\g)}+\p_\g x^o_{-}(\g, \eta)\big(\frac{1}{B_\eps(\g)}-\frac{1}{B}\big)\\
&\qquad+\p_\g x^o_\eps(\g-\eta)\big(\frac{1}{B_\eps(\g)}-\frac{1}{B_\eps(\g-\eta)}\big)\Big]\cdot\frac{\dg^2x^o_\varepsilon(\g-\e)}{|\e|^{1+s}}d\e:=G_{5, 1}+G_{5, 2}+G_{5,3},
\end{align*}
where the splitting is according to the terms in the square brackets. Arguing as in $G_6$ gives
\begin{align*}
&\| G_{5, 1}\|_{L^2}\le \frac{C}{LB}\| \p_\g x^o_\eps-\p_\g x^o\|_{C^\mez}\| \p_\g^2x^o_\eps\|_{L^2},\\
&\| G_{5, 2}\|_{L^2}\le \frac{C}{L}\|B_\eps^{-1}(\cdot)-B^{-1}\|_{L^\infty}\| \p_\g x^o\|_{C^\mez}\| \p_\g^2x^o_\eps\|_{L^2},
\end{align*}
where both right-hand sides tend to $0$. For $G_{5, 3}$ we note that 
\begin{align*}
&|B_\eps(\g)-B_\eps(\g-\eta)|\le C\| \p_\g x^o_\eps\|_{C^{\mez+s}}|\e|^{\mez+s},\\
&|B_\eps(\g)-B_\eps(\g-\eta)|=|B_\eps(\g)-B(\g)+B(\g-\eta)-B_\eps(\g-\eta)|\le 2\| B_\eps(\cdot)-B\|_{L^\infty},
\end{align*}
whence 
\[
|B_\eps(\g)-B_\eps(\g-\eta)|\le C|\e|^{\mez}\| \p_\g x^o_\eps\|_{C^{\mez+s}}^{\frac{1}{1+2s}}\| B_\eps(\cdot)-B\|_{L^\infty}^{\frac{2s}{1+2s}}.
\]
Then using the argument in $G_6$  yields
\[
\|G_{5, 3}\|_{L^2}\le \frac{C}{LB^2}\| \p_\g x^o_\eps\|_{C^{\mez+s}}^{\frac{1}{1+2s}}\| B_\eps(\cdot)-B\|_{L^\infty}^{\frac{2s}{1+2s}}\| \p_\g^2x^o_\eps\|_{L^2}
\]
which again tends to $0$. This concludes the proof of \eqref{phiproperties}.
	\appendix
	\section{Littlewood-Payley theory}
	Let $\chi:\Rr^d\to \Rr$ be $C^\infty$ such that $\chi(\xi)=1$ for $|\xi|\le \frac12$ and $\chi(\xi)=0$ for $|\xi|\ge 1$. Set 
	\[
	\varphi(\xi)=\chi(\frac{\xi}{2})-\chi(\xi).
	\]
	For $k\ge 0$ we define 
	\[
	\chi_k(\xi)=\chi(\frac{\xi}{2^k}),\quad \varphi_k(\xi)=\chi_k(\xi)-\chi_{k-1}(\xi)=\varphi(\frac{\xi}{2^k}),
	\]
	so that $\varphi_k$ is supported in the annulus $\{2^{k-1}<|\xi|<2^{k+1}\}$. Clearly
	\[
	\chi(\xi)+\sum_{k=0}^\infty \varphi_k(\xi)=1\quad\forall \xi\in \Rr^d.
	\]
	For $f:\T^d\to \Rr$ and $j\ge 0$ we define the Fourier multipliers 
	\begin{align*}
		&\varphi_{-1}(D)f=\chi(D)f,\quad \Delta_j f=\cF^{-1}(\varphi_j\cF(f))\quad\forall j\ge 0,\\
		&\quad S_jf=\sum_{-1\le k\le j-1}\Delta_k f=\chi(2^{-j+1}D)f\quad\forall j\ge 0.
	\end{align*}
	We have $\varphi_{-1}(D)f=\wh{f}(0)$ and (formally)
	\[
	f=\sum_{j\ge -1} \Delta_j f=\wh{f}(0)+\sum_{j\ge 0} \Delta_j f.
	\]
	For $p, r\in [1, \infty]$ and $s\in \Rr$ we define the inhomogeneous Besov norm by 
	\bq
	\| f\|_{B^s_{p, r}(\T^d)}=\| 2^{sj}\| \Delta_j f\|_{L^p(\T^d)}\|_{\ell^r(\{-1\}\cup \Nn)},\quad \Nn=\{0, 1, \dots\}.
	\eq
	It is well-known that the Besov space $B^s_{2, 2}(\T^d)$ coincides with the  Sobolev space  $H^s(\T^d)$ and 
	\[
	\| f\|_{B^s_{p, q}(\T^d)}\simeq \| f\|_{H^s(\T^d)}.
	\]
	The homogeneous Besov norm is defined by removing $\Delta_{-1}f=\wh{f}(0)$, i.e.
	\bq
	\| f\|_{\dot B^s_{p, r}(\T^d)}=\| 2^{sj}\| \Delta_j f\|_{L^p(\T^d)}\|_{\ell^r( \Nn)}.
	\eq
	\begin{prop}\label{equi:HZ}
		For $\alpha\in (0, 1)$, the $\alpha$-H\"older (semi)norm is defined by 
		\bq
		\| f\|_{\dot C^\alpha(\T)}=\sup_{x,~y\in \T^d,~x\ne y}\frac{|f(x)-f(y)|}{|x-y|^\alpha}.
		\eq
		The norms $\| \cdot \|_{\dot C^\alpha(\T^d)}$ and $\| \cdot \|_{\dot B^\alpha_{\infty, \infty}(\T^d)}$  are equivalent. Consequently, the norms $\| \cdot \|_{C^\alpha(\T^d)}$ and $\| \cdot \|_{\dot B^\alpha_{\infty, \infty}(\T^d)}$  are equivalent, where
		\bq
		\| \cdot \|_{C^\alpha(\T^d)}=\| \cdot \|_{\dot C^\alpha(\T^d)}+\| \cdot\|_{L^\infty(\T^d)}.
		\eq
	\end{prop}
	\begin{proof}
		Let $x\ne y\in \T^d$ and choose $j_0\ge -1$ such that $|x-y|\simeq 2^{-j_0}$. Clearly $\Delta_{-1}f(x)-\Delta_{-1}f(y)=0$. For $0\le j\le j_0$ we use Bernstein's inequality to obtain
		\[
		|\Delta_j f(x)-\Delta_jf(y)|\le |x-y|\| \na \Delta_j f\|_{L^\infty}\les  2^j|\| \Delta_j f\|_{L^\infty}|x-y|\les  2^{j(1-\alpha)}\| f\|_{\dot B^\alpha_{\infty, \infty}(\T^d)}|x-y|.
		\]
		On the other hand, when $j\ge j_0$ we simply estimate 
		\[
		|\Delta_j f(x)-\Delta_jf(y)|\le 2\|\Delta_j f\|_{L^\infty}\le 22^{-j\alpha}\| f\|_{\dot B^\alpha_{\infty, \infty}(\T^d)}.
		\]
		Summing over $j\ge -1$ gives
		\begin{align*}
			|f(x)-f(y)|&\les \sum_{j=0}^{j_0}2^{j(1-\alpha)}\| f\|_{\dot B^\alpha_{\infty, \infty}(\T^d)}|x-y|+\sum_{j=j_0+1}^\infty 2^{-j\alpha}\| f\|_{\dot B^\alpha_{\infty, \infty}(\T^d)}\\
			&\les 2^{j_0(1-\alpha)}\| f\|_{\dot B^\alpha_{\infty, \infty}(\T^d)}|x-y|+2^{-j_0\alpha}\| f\|_{\dot B^\alpha_{\infty, \infty}(\T^d)}\\
			&\les |x-y|^{-1+\alpha}\| f\|_{\dot B^\alpha_{\infty, \infty}(\T^d)}|x-y|+|x-y|^{\alpha}\| f\|_{\dot B^\alpha_{\infty, \infty}(\T^d)}\\
			&\les |x-y|^{\alpha}\| f\|_{\dot B^\alpha_{\infty, \infty}(\T^d)}.
		\end{align*}
		Thus $\| f\|_{\dot C^\alpha}\les \| f\|_{\dot B^\alpha_{\infty, \infty}(\T^d)}$. 
		
		Conversely, for $j\ge 0$ we have $\int_{\T^d} \cF^{-1}\varphi_j(x)dx=0=\varphi_j(0)=0$ and thus
		\begin{align*}
			|\Delta_j f(x)|&=\Big|\int_{\T^d}\cF^{-1}(\varphi_j)(x-y)(f(y)-f(x)) dy\Big|\\
			&=\Big|\int_{\T^d}\sum_{k\in \Zz^d}\cF_{\Rr^d}^{-1}(\varphi_j)(x-y+2\pi k)(f(y)-f(x))dy \Big|\\
			&=\Big|\int_{\Rr^d}\cF_{\Rr^d}^{-1}(\varphi_j)(x-y)(f(y)-f(x))dy \Big|,
		\end{align*} 
		where we have used the Poisson summation 
		\[
		\cF^{-1}(\tt)(x)=\sum_{k\in \Zz^d}\cF_{\Rr^d}^{-1}(\tt)(x+2\pi k)
		\]
		for all Schwartz function $\tt:\Rr^d\to \Cc$. Here $\cF^{-1}_{\Rr^d}\tt(x)=(2\pi)^{-d}\int_{\Rr^d}e^{ix\xi}\tt(\xi)d\xi$.  It follows that 
		\[
		|\Delta_j f(x)|\le \| f\|_{\dot C^\alpha}\Big|\int_{\Rr^d}|x-y|^\alpha\cF_{\Rr^d}^{-1}(\varphi_j)(x-y)dy \Big|\les  2^{-j\alpha}\| f\|_{\dot C^\alpha}
		\]
		and thus $ \| f\|_{\dot B^\alpha_{\infty, \infty}(\T^d)}\les\| f\|_{\dot C^\alpha}$. We have proved the equivalence between $\| \cdot \|_{\dot C^\alpha(\T^d)}$ and $\| \cdot \|_{\dot B^\alpha_{\infty, \infty}(\T^d)}$. Regarding the equivalence between the inhomogeneous norms, it suffices to note that 
		\[
		\|\Delta_{-1}f\|_{L^\infty}\les \| f\|_{L^\infty},\quad \| f\|_{L^\infty}\le \sum_{j\ge -1}\| \Delta_j f\|_{L^\infty}\les \| f\|_{B^\alpha_{\infty, \infty}}.
		\]
	\end{proof}
	\begin{defi}
		For $a, u:\T^d\to \Cc$ the paraproduct $T_au$ is defined by 
		\[
		T_au=\sum_{j\ge 1}S_{j-1}a\Delta_ju.
		\]
		The (formal) Bony decomposition is given by
		\[
		au=T_au+T_ua+R(a, u),
		\]
		where 
		\[
		R(a, u)=\sum_{j, k\ge -1,~ |j-k|\le 1}\Delta_j a\Delta_k u.
		\]
	\end{defi}
	Note that $\supp{\cF(S_{j-1}a\Delta_j u)}\subset \{2^{j-1}\le |\xi|\le 2^{j+2}\}$ for all $j\ge 1$.
	\begin{prop}
		(i) If $s_1, s_2\in \Rr$ such that $s_1+s_2>0$ then 
		\bq\label{BonyR}
		\| R(a, u)\|_{H^{s_1+s_2}}\les \| a\|_{B^{s_1}_{\infty, \infty}}\| u\|_{H^{s_2}}.
		\eq
		(ii) For all $s\in \Rr$ and $m>0$
		\begin{align}\label{pp:1}
			\| T_a u\|_{H^s}&\les \| a\|_{L^\infty}\| u\|_{H^s},\\ \label{pp:2}
			\| T_a u\|_{H^s}&\les \| a\|_{B^{-m}_{\infty, \infty}}\| u\|_{H^{s+m}}.
		\end{align}
		(ii) For all $s\in \Rr$ and $r>0$, 
		\bq\label{pp:new1}
		\| T_au\|_{H^s}\les \|a\|_{H^{-r}}\| u\|_{B^{s+r}_{\infty, 2}}.
		\eq
	\end{prop}
	\begin{proof}
		The inequalities \eqref{BonyR}, \eqref{pp:1} and \eqref{pp:2} can be found in Section 2.8, \cite{BCD}. Let us prove \eqref{pp:new1}. For all $j\ge 1$, 
		\begin{align*}
			\| S_{j-1}a \Delta_j u\|_{L^2}&\le \sum_{k=-1}^{j-2}\| \Delta_k a\|_{L^2}\| \Delta_j u\|_{L^\infty}\\
			&\le  \Big(\sum_{k=-1}^{j-2}2^{-2rk}\| \Delta_k a\|^2_{L^2}\Big)^\mez \Big(\sum_{k=-1}^{j-2}2^{2rk}\Big)^\mez\| \Delta_j u\|_{L^\infty}\\
			&\les \| a\|_{H^{-r}}2^{rj}\| \Delta_j u\|_{L^\infty}.
		\end{align*}
		Consequently, 
		\[
		\| T_au\|_{H^s}^2\les\| a\|_{H^{-r}}^2\sum_{j\ge 1}2^{2(s+r)j}\| \Delta_j u\|^2_{L^\infty}\les \| a\|_{H^{-r}}^2\| u\|_{B^{s+r}_{\infty, 2}}^2.
		\]
	\end{proof}
	\begin{lemm}[\protect{\cite[Lemma~2.99]{BCD}}]\label{comt:LdT}
		For $s, m\in \Rr$ and $p, r\in [1, \infty]$, there exists $C>0$ such that
		\bq\label{cmt:pp}
		\| [\Ld^m, T_a] u\|_{B^s_{p, r}(\T^d)}\le C\| \na a\|_{L^\infty(\T^d)}\| u\|_{B^{s+m-1}_{p, r}(\T^d)}.
		\eq
	\end{lemm}

\vspace{.1in}	
\noindent{\bf{Acknowledgment.}} 
F. Gancedo was partially supported by the ERC through the Starting Grant project H2020-EU.1.1.-639227 and by the grant EUR2020-112271 (Spain).  H.Q. Nguyen was partially supported by NSF grant DMS-190777. N. Patel was partially supported by the AMS-Simons Travel Grants, which are administered by the American Mathematical Society with support from the Simons Foundation.

\end{document}